\documentclass[11pt]{amsart}
\usepackage{fullpage}
\usepackage{amssymb,amsmath,amsfonts,amsthm,mathrsfs}
\usepackage{color}
\definecolor{webcolor}{rgb}{0.8,0,0.2}
\definecolor{webbrown}{rgb}{.6,0,0}
\usepackage[
        colorlinks,
        linkcolor=webbrown,  filecolor=webcolor,  citecolor=webbrown, 
        backref,
        pdfauthor={David Zywina}, 
       pdftitle={Computing actions on cusp forms}, 
]{hyperref}
\usepackage[alphabetic,backrefs,lite]{amsrefs} 

\numberwithin{equation}{section}

\newcommand{\CC}{\mathbb C}

\newcommand{\PP}{\mathbb P}
\newcommand{\QQ}{\mathbb Q}
\newcommand{\RR}{\mathbb R}

\newcommand{\ZZ}{\mathbb Z}

\newcommand{\OO}{\mathcal O}
  \newcommand{\calF}{\mathcal F}

\newcommand{\calN}{\mathcal N}

\newcommand{\calR}{\mathcal R}

\newcommand{\calX}{\mathcal X}

\newcommand{\calM}{\mathcal M}
\newcommand{\calT}{\mathcal T}

\newcommand{\p}{\mathfrak p}

\def\new{{\operatorname{new}}}

\def\old{{\operatorname{old}}}

\newcommand{\brak}[1]{[\![ #1 ]\!]}

\newcommand{\ang}[1]{\langle #1 \rangle}

\def\Spec{\operatorname{Spec}} 
\def\Gal{\operatorname{Gal}}
 
\def \GL {\operatorname{GL}}

\def \SL {\operatorname{SL}}

\def\Aut{\operatorname{Aut}}

\def\Tr{\operatorname{Tr}}

\newcommand{\defi}[1]{\textsf{#1}} 

\newcommand\blank[1]{}

\def\bbar#1{\setbox0=\hbox{$#1$}\dimen0=.2\ht0 \kern\dimen0 
\overline{\kern-\dimen0 #1}}
\newcommand{\Qbar}{{\overline{\mathbb Q}}}

\newcommand{\kbar}{\bbar{k}}

\newtheorem{thm}{Theorem}[section]
\newtheorem{lemma}[thm]{Lemma}
\newtheorem{cor}[thm]{Corollary}
\newtheorem{prop}[thm]{Proposition}

\theoremstyle{definition}

\theoremstyle{remark}
\newtheorem{remark}[thm]{Remark}
\newtheorem{example}[thm]{Example}

\newenvironment{romanenum}{\hfill \begin{enumerate} }{\end{enumerate}}
\newenvironment{alphenum}{\hfill \begin{enumerate} }{\end{enumerate}}

\usepackage{ebgaramond}

\newcommand\remove[1]{}

\begin{document}
\title{Computing actions on cusp forms}

\author{David Zywina}
\address{Department of Mathematics, Cornell University, Ithaca, NY 14853, USA}
\email{zywina@math.cornell.edu}
\urladdr{http://www.math.cornell.edu/~zywina}

\date{{\today}}

\begin{abstract}
For positive integers $k$ and $N$, we describe how to compute the natural action of $\SL_2(\ZZ)$ on the space of cusp forms $S_k(\Gamma(N))$, where a cusp form is given by sufficiently many terms of its $q$-expansion.    This will reduce to computing the action of the Atkin--Lehner operator on $S_k(\Gamma)$ for a congruence subgroup $\Gamma_1(N)\subseteq \Gamma \subseteq \Gamma_0(N)$.    Our motivating application of such fundamental computations is to compute explicit models of some modular curves $X_G$.
\end{abstract}



\maketitle

\section{Introduction}

Fix positive integers $k$ and $N$, and a congruence subgroup $\Gamma_1(N) \subseteq \Gamma \subseteq \Gamma_0(N)$.  Let $S_k(\Gamma)$ and $M_k(\Gamma)$ be the space of cusp forms and modular forms, respectively, of weight $k$ with respect to $\Gamma$.   In this article we explain how one can explicitly compute the action of the Atkin--Lehner involution $W_N$ on $S_k(\Gamma)$ and $M_k(\Gamma)$, where we view a modular form as being given by its $q$-expansion with enough terms known to uniquely determine it (given $k$ and $\Gamma$).  In \S\ref{SS:SL2 action}, we will explain how this allows us to compute the natural right action of $\SL_2(\ZZ)$ on $S_k(\Gamma(N))$.    In \S\ref{SS:intro modular curves}, we give an application to computing models of modular curves.  In \S\ref{SS:related}, we recall some related results.

\subsection{Background and notation}

We recall some basic definitions and conventions.  Fix a positive integer $k$.

The group $\GL_2^+(\RR)$ of $2\times 2$ real matrices with positive determinant acts on the complex upper half plane $\mathfrak{H}$ by linear fractional transformations.   For a function $f\colon \mathfrak{H}\to \CC$ and a matrix $\gamma= \left(\begin{smallmatrix}a & b \\c & d\end{smallmatrix}\right) \in \GL_2^+(\RR)$, we define the function $f|_k\gamma\colon \mathfrak{H}\to\CC$, $\tau\mapsto \det(\gamma)^{k/2} (c\tau+d)^{-k} f(\gamma\tau)$.  We will simply write $f|\gamma$ when $k$ is fixed or clear from context.   We have $f|(\gamma\gamma')=(f|\gamma)|{\gamma'}$ for all $\gamma,\gamma'\in \GL_2^+(\RR)$.    

For a congruence subgroup $\Gamma$, we denote the space of cusp and modular forms by $S_k(\Gamma)$ and $M_k(\Gamma)$, respectively.  They consist of holomorphic functions $f\colon \mathfrak{H} \to \CC$ that satisfy $f|\gamma=f$ for all $\gamma\in \Gamma$ along with usual conditions at the cusps.    Let $w$ be the width of the cusp of $\Gamma$ at $\infty$, i.e., the smallest positive integer $w\geq 1$ for which $\left(\begin{smallmatrix}1 & w \\0 & 1\end{smallmatrix}\right)$ lies in $\Gamma$.  For each modular form $f\in M_k(\Gamma)$, we have
\[
f(\tau) = \sum_{n=0}^\infty a_n(f)\, q_w^{n}
\]
for unique $a_n(f) \in \CC$, where $q_w:=e^{2\pi i \tau/w}$; this is the \defi{Fourier series} or  \defi{$q$-expansion} of $f$.  When $w=1$, we will simply write $q$ for $q_1$.   For a subring of $R$ of $\CC$, we denote by $M_k(\Gamma,R)$ and $S_k(\Gamma,R)$ the $R$-module consisting of modular forms $f$ in $M_k(\Gamma)$ and $S_k(\Gamma)$, respectively, for which all of the coefficients $a_n(f)$ lie in $R$. \\

Fix a positive integer $N$ and a congruence subgroup $\Gamma_0(N) \subseteq \Gamma \subseteq \Gamma_1(N)$.  Since the matrix $\left(\begin{smallmatrix}0 & -1 \\N & 0\end{smallmatrix}\right)$ normalizes $\Gamma$, we obtain an automorphism 
\[
M_k(\Gamma) \xrightarrow{\sim} M_k(\Gamma),\quad f\mapsto N^{k/2} \cdot f| \left(\begin{smallmatrix}0 & -1 \\N & 0\end{smallmatrix}\right) =: f|W_N
\]
which we will call the \defi{Atkin--Lehner operator} of $M_k(\Gamma)$.   We have $(f|W_N)(\tau)=\tau^{-k} f(-1/(N\tau))$ and 
\[
(f|W_N)|W_N= (-1)^k N^k\cdot f.
\] 
Note that the subspace $S_k(\Gamma)$ is stable under the action of $W_N$.
\begin{remark}
In the literature, our operator $W_N$ is often scaled by a factor of $N^{-k/2}$; in particular, it would then be an involution when $k$ is even.  Our version has nicer arithmetic properties when $k$ is odd.  For example, $M_k(\Gamma_0(N),\QQ)$ is always stable under the action of $W_N$ using our normalization.  
\end{remark}

For $d\in (\ZZ/N\ZZ)^\times$, the \defi{diamond operator} $\ang{d}$ acts on $M_k(\Gamma)$ and $S_k(\Gamma)$; we have $f|\ang{d} := f|\gamma$, where $\gamma\in \SL_2(\ZZ)$ is any matrix satisfying $\gamma\equiv \left(\begin{smallmatrix}d^{-1} & * \\0 & d\end{smallmatrix}\right) \pmod{N}$.

Take any automorphism $\sigma$ of the field $\CC$ and any modular form $f\in M_k(\Gamma(N))$.  Let $\sigma(f)$ be the modular form in $M_k(\Gamma(N))$ whose $q$-expansion is obtained by applying $\sigma$ to the coefficients of the $q$-expansion of $f$.   This defines an action of $\Aut(\CC)$ on $M_k(\Gamma(N))$.   The subspaces $M_k(\Gamma)$ and $S_k(\Gamma)$ are stable under the action of $\Aut(\CC)$, where $\Gamma=\Gamma(N)$ or $\Gamma_0(N) \subseteq \Gamma \subseteq \Gamma_1(N)$.  

Finally, for each positive integer $N$, we define the $N$-th root of unity $\zeta_N:= e^{2\pi i /N} \in \CC^\times$.

\subsection{Setup} \label{SS:setup intro}

Fix positive integers $k$ and $N$.  Let $\calM$ be a subspace of $M_k(\Gamma_1(N))$.   For each subring $R$ of $\CC$, we define $\calM(R):=\calM \cap M_k(\Gamma_1(N),R)$, i.e., the $R$-submodule of $\calM$ consisting of modular forms whose Fourier coefficients all lie in $R$.  

Assume that $\calM$ satisfies all the following conditions:
\begin{alphenum}
\item \label{c:a}
$\calM$ is stable under the action of $W_N$.
\item \label{c:b}
$\calM$ is stable under the action of the diamond operators $\ang{d}$ with $d\in (\ZZ/N\ZZ)^\times$,
\item \label{c:c}
$\calM(\ZZ)$ spans $\calM$ as a $\CC$-vector space,
\item \label{c:d}
the $\ZZ$-module $\calM(\ZZ)$ has a basis $f_1,\ldots, f_g$, where each $f_j$ is given by its $q$-expansion for which we can compute an arbitrary number of terms. 
\end{alphenum}

By our assumptions, $f_1,\ldots, f_g$ is a basis of $\calM$.  The main goal of this paper is explain how to compute the action of the Atkin--Lehner operator $W_N$ on the space $\calM$ with respect to a fixed basis $f_1,\ldots, f_g$.   Equivalently, we will give an algorithm to compute the unique matrix $W \in \GL_g(\CC)$ satisfying
\[
f_j|W_N = \sum_{k=1}^g  W_{j,k}\cdot f_k
\]
for all $1\leq j \leq g$.   We will see that all the entries of $W$ lie in the cyclotomic field $\QQ(\zeta_N)$.

We are motivated by the following examples of spaces $\calM$.

\begin{example}
The spaces $M_k(\Gamma)$ and $S_k(\Gamma)$, where $\Gamma$ is a congruence subgroup with $\Gamma_1(N) \subseteq \Gamma \subseteq \Gamma_0(N)$, satisfy conditions (\ref{c:a})--(\ref{c:d}).  We shall verify these conditions in \S\ref{S:numerical}.
\end{example}

\begin{example} \label{Ex:newform M}
Fix a newform $f\in S_k(\Gamma_1(N))$ of level $N$.     Let $\calM_f$ be the $\CC$-subspace of $S_k(\Gamma_1(N))$ generated by $\sigma(f)$ with $\sigma\in \Aut(\CC)$.  In \S\ref{SS:Mf}, we will verify that conditions (\ref{c:a})--(\ref{c:d}) hold for $\calM_f$.  There is a unique $c\in \CC$ with absolute value $N^{k/2}$ satisfying
\[
f|W_N = c \, \bbar{f},
\]
where $\bbar{f}$ is obtained by applying complex conjugation to the coefficients of the $q$-expansion of $f$, cf.~Proposition~\ref{P:W action on newforms}.  After expressing $f$ as a linear combination of the cusp forms $f_1,\ldots,f_g$, one can use the matrix $W$ to compute the exact value of $c$.
\end{example}

\subsection{The algorithm} \label{SS:the algorithm}

Fix notation and assumptions as in \S\ref{SS:setup intro}.   We now describe our algorithm to compute the matrix $W$; its validity will be proved in \S\ref{S:verification}.   The main idea is to use numerical approximations of $W$ to compute the true value.
\footnote{See \url{http://pi.math.cornell.edu/~zywina/papers/AtkinLehner/} for a basic implementation, in \texttt{Magma}, of the algorithms in this paper.}\\

For a fixed integer $n\geq 1$, let $A$ be the $g \times n$ matrix satisfying $A_{i,j}=a_{j-1}(f_i)$, where $f_i = \sum_{j=0}^\infty a_j(f_i) q^j$.  By taking $n$ large enough and using that $f_1,\ldots, f_g$ is a basis of $\calM(\ZZ)$, we may assume that $A\in M_{g,n}(\ZZ)$ has rank $g$.   

Recall that a matrix in $M_{g,n}(\ZZ)$ is in \defi{Hermite normal form} if it satisfies all the following conditions:
\begin{itemize}
\item
it is upper triangular and its zero rows lie below all the non-zero rows,
\item
the \emph{pivot} in each non-zero row, i.e., the first non-zero entry, is positive and strictly to the right of all pivots in rows that lie above it,
\item
the entries of the matrix above a pivot are all non-negative and smaller than the pivot.
\end{itemize}
There is a matrix $U \in \SL_g(\ZZ)$ such that $H:=UA$ is in Hermite normal form.   The matrix $H$, though not $U$, is uniquely determined.   With a change of basis of $\calM(\ZZ)$ given by $U$, one could take the basis $f_1,\ldots, f_g$ so that $A$ is in Hermite normal form; this would give a distinguished basis of $\calM(\ZZ)$ by the uniqueness of $H$.  We define $\alpha$ to be the product of all the pivots of $H$.

  Let $Q$ be the smallest positive divisor $N$ for which the action of the diamond operators $\ang{d}$ on $\calM$ depends only on the value of $d$ modulo $Q$.      The diamond operators thus give an action of $(\ZZ/Q\ZZ)^\times$ on $\calM$.   For each $d\in (\ZZ/Q\ZZ)^\times$, let 
  $D_d \in M_g(\CC)$ be the matrix that satisfies $f_j|\ang{d} = \sum_{k=1}^g (D_d)_{j,k} \cdot f_k$ for all $1\leq j\leq g$.  Moreover, we will see later that $D_d \in \GL_g(\ZZ)$.   \\
  
For an integer $0 \leq b \leq \varphi(Q)-1$, where $\varphi$ is the Euler totient function, define the matrix
\[
\beta_b:= W \cdot \sum_{d\in (\ZZ/Q\ZZ)^\times} \zeta_Q^{db} \, D_d.
\]
In \S\ref{S:verification}, we will prove that the matrix $W$ lies in $M_g(\QQ(\zeta_Q))$ and satisfies $\Tr_{\QQ(\zeta_Q)/\QQ}(\zeta_Q^b W) = \beta_b$, where the trace of a matrix is taken entry by entry.  In particular, $\beta_b \in M_g(\QQ)$.  Define the integer 
 \[
B_{k,N}:= \prod_{p|N} p^{\lceil  k/(p-1)\rceil},
\]
where $\lceil{x}\rceil$ denotes $x$ rounded up to the nearest integer.  We will prove that $B_{k,N}\, \alpha \cdot W$ lies in $M_g(\ZZ[\zeta_Q])$ and hence $B_{k,N} \,\alpha \cdot \beta_b \in M_g(\ZZ)$.

In \S\ref{SS:numerical}, we will observe that $W$ and $D_d$ can be numerically approximated in $M_g(\CC)$.  By approximating $D_d\in \GL_g(\ZZ)$ to a sufficiently high accuracy, we can compute $D_d$.   By computing $W$ to a sufficiently high accuracy, we can approximate the entries of $B_{k,N} \alpha \cdot \beta_b \in M_g(\ZZ)$ so that they can be explicitly determined.  With this approach, we are now able to compute the matrices $\beta_b \in M_g(\QQ(\zeta_N))$  for all $0\leq b \leq \varphi(N)-1$.   \\

Finally observe that $W$ is the unique matrix in $M_g(\QQ(\zeta_Q))$ satisfying $\Tr_{\QQ(\zeta_Q)/\QQ}(\zeta_Q^b W)= \beta_b$ for all $0\leq b \leq \varphi(Q)-1$.  Indeed, note that the map
\[
\calT \colon \QQ(\zeta_Q) \to \QQ^{\varphi(Q)},\quad x\mapsto (\Tr_{\QQ(\zeta_Q)/\QQ}(\zeta_Q^b x))_{0 \leq b \leq \varphi(Q)-1}
\]
is an isomorphism of $\QQ$-vector spaces (the pairing $\QQ(\zeta_Q)\times \QQ(\zeta_Q) \to \QQ$, $(x,y)\mapsto \Tr_{\QQ(\zeta_Q)/\QQ}(xy)$ is non-degenerate and $1,\zeta_Q,\ldots, \zeta_Q^{\varphi(Q)-1}$ is a basis for $\QQ(\zeta_Q)$ over $\QQ$).     By computing $\Tr_{\QQ(\zeta_Q)/\QQ}(\zeta_Q^b \zeta_Q^a)$ with $0\leq a,b \leq \varphi(Q)-1$, we can compute $\calT$ and its inverse.    Using this, we can compute the desired matrix $W$.

\subsection{Action of $\SL_2(\ZZ)$ on $S_k(\Gamma(N))$}
\label{SS:SL2 action}

Fix positive integers $k$ and $N$.     We now explain how, using the algorithm of \S\ref{SS:the algorithm}, we can compute the natural right action of $\SL_2(\ZZ)$ on $S_k(\Gamma(N))$.  

The matrices $S:=\left(\begin{smallmatrix}0 & -1 \\1 & 0\end{smallmatrix}\right)$ and $T:=\left(\begin{smallmatrix}1 & 1 \\0 & 1\end{smallmatrix}\right)$ generate $\SL_2(\ZZ)$, so to describe the action of $\SL_2(\ZZ)$ on $S_k(\Gamma(N))$ it suffices to describe how $S$ and $T$ act.   The action of $T$ is straightforward since it fixes the cusp at infinity.  For any $h=\sum_{n=1}^\infty a_n(h) q_N^{n} \in S_k(\Gamma(N))$, with $q_N=e^{2\pi i \tau/N}$, we have
\[
h | T = \sum_{n=1}^\infty a_n(h) \zeta_N^n \cdot q_N^{n}.
\]

Define $\Gamma:=\Gamma_0(N^2)\cap \Gamma_1(N)$.   Using the algorithm of \S\ref{SS:the algorithm} with level $N^2$ instead of $N$, we can compute an explicit basis $f_1,\ldots, f_g$ of the $\ZZ$-module $S_k(\Gamma,\ZZ)$ and a matrix $W\in  \GL_g(\QQ(\zeta_N))$ satisfying $f_j|W_{N^2} = \sum_{k=1}^g  W_{j,k}\cdot f_k$.   Since $\Gamma(N)=\left(\begin{smallmatrix}N & 0 \\0 & 1\end{smallmatrix}\right) \Gamma \left(\begin{smallmatrix}N & 0 \\0 & 1\end{smallmatrix}\right)^{-1}$, we have an isomorphism 
\[
\beta\colon S_k(\Gamma) \xrightarrow{\sim}  S_k(\Gamma(N)),\quad f\mapsto N^{k/2} \cdot f|\left(\begin{smallmatrix}1 & 0 \\0 & N\end{smallmatrix}\right)
\]
of complex vector spaces which on $q$-expansions satisfies $\beta(\sum_{n=1}^\infty a_n q^n)=\sum_{n=1}^\infty a_n q_N^{n}$.  For each $1\leq j \leq g$, define $h_j:=\beta(f_j)=\sum_{n=1}^\infty a_n(f_j) q_N^{n}$.   The cusp forms $h_1,\ldots, h_g$ are a basis of the $\ZZ$-module $S_k(\Gamma(N),\ZZ)$.

For any $f\in S_k(\Gamma)$, we have
\[
\beta(f|W_{N^2}) = N^k \cdot \beta(f) |\big(\left(\begin{smallmatrix}1 & 0 \\0 & N\end{smallmatrix}\right)^{-1} \left(\begin{smallmatrix}0 & -1 \\N^2 & 0\end{smallmatrix}\right)  \left(\begin{smallmatrix}1 & 0 \\0 & N\end{smallmatrix}\right) \big)= N^k \cdot \beta(f) | \left(\begin{smallmatrix}0 & -N \\N & 0\end{smallmatrix}\right)=N^k \cdot \beta(f)|S.
\]
Therefore, 
\[
h_j | S = \beta(f_j)|S = N^{-k} \beta(f_j|W_{N^2}) = N^{-k} \beta(\sum_{k=1}^g W_{j,k} \cdot f_k) = \sum_{k=1}^g N^{-k} W_{j,k} \cdot \beta(f_k) = \sum_{k=1}^g N^{-k} W_{j,k} \cdot h_k.
\]
So the action of $S$ on the basis $h_1,\ldots, h_g$ of $S_k(\Gamma(N))$ is given by the matrix $W$.  In \S\ref{SS:the algorithm}, we gave an algorithm to compute $W$. 

\subsection{Modular curves} \label{SS:intro modular curves}

Fix an integer $N>1$ and let $G$ be a subgroup of $\GL_2(\ZZ/N\ZZ)$ that satisfies $\det(G)=(\ZZ/N\ZZ)^\times$ and $-I\in G$.   Associated to the group $G$, is a \defi{modular curve} $X_G$; it is a smooth projective and geometrically irreducible curve $X_G$ defined over $\QQ$.    In \S\ref{S:modular curves}, we give a definition of $X_G$ and also give a connection with elliptic curves.

There is a natural right action of $\GL_2(\ZZ/N\ZZ)$ on the $\QQ$-vector space $S_2(\Gamma(N),\QQ(\zeta_N))$ characterized by the following properties: 
\begin{itemize} 
\item
 The group $\SL_2(\ZZ/N\ZZ)$ acts via the right action of $\SL_2(\ZZ)$ described in \S\ref{SS:SL2 action}.   
 \item
 A matrix $\left(\begin{smallmatrix}1 & 0 \\ 0 & d\end{smallmatrix}\right)$ acts on a cusp form by applying $\sigma_d$ to the coefficients of its $q$-expansion, where $\sigma_d$ is the automorphism of $\QQ(\zeta_N)$ satisfying $\sigma_d(\zeta_N)=\zeta_N^d$.
\end{itemize}
See \S3 of \cite{BN2019} for an explanation of why this action is well-defined.

From \S\ref{SS:SL2 action} and our algorithm in \S\ref{SS:the algorithm}, we can compute a basis $f_1,\ldots, f_g$ of the $\QQ$-vector space $S_2(\Gamma(N),\QQ(\zeta_N))^G$, where each $f_j$ is given by its $q$-expansion for which we can compute arbitrarily many terms.    In \S\ref{S:modular curves}, we will see that $S_2(\Gamma(N),\QQ(\zeta_N))^G$ is naturally isomorphism to $H^0(X_G,\Omega_{X_G})$.  Let $\omega_1,\ldots, \omega_g$ be the basis of $H^0(X_G,\Omega_{X_G})$ corresponding to $f_1,\ldots, f_g$.   In particular, the genus of $X_G$ is $g$.

Now assume that $g\geq 2$.  The morphism 
\[
\varphi\colon X_G \to \PP^{g-1}_\QQ,\quad P\mapsto [\omega_1(P),\ldots, \omega_g(P)]
\]
is called the \defi{canonical map} and it is uniquely determined up to an automorphism of $\PP^{g-1}_\QQ$. The image $C:=\varphi(X_G)$ is the \defi{canonical curve} of $X_G$.

We will see that the homogenous ideal $I(C)\subseteq \QQ[x_1,\ldots, x_g]$ of the curve $C$ is generated by the homogeneous polynomials $F \in \QQ[x_1,\ldots, x_g]$ for which $F(f_1,\ldots, f_g)=0$.   In \S\ref{S:modular curves}, we describe how to find a set of generators of the ideal $I(C)$ from enough terms of the $q$-expansions of the cusp forms $f_1,\ldots, f_g$, and hence compute the curve $C$.

If $X_G$ is (geometrically) hyperelliptic, then $C$ has genus $0$ and $\varphi$ has degree $2$.  If $X_G$ is not hyperelliptic, then $\varphi$ is an embedding and hence $X_G$ and $C$ are isomorphic curves.

\begin{remark}
The bottleneck in the above approach to computing modular curves is that $S_2(\Gamma(N),\QQ(\zeta_N))$ with its $\SL_2(\ZZ)$-action becomes harder to compute as $N$ grows.  In particular, note that $S_2(\Gamma(N),\QQ(\zeta_N))^G$ often has much smaller dimension than $S_2(\Gamma(N),\QQ(\zeta_N))$.   The original goal of this paper was to show that this obvious and direct approach is actually computable.
\end{remark}

\subsection{Examples}

We now give a few basic examples.

 \begin{example}
Define the congruence subgroup $\Gamma:=\Gamma_0(49) \cap \Gamma_1(7)$; it has level $N=49$.  Then there is a unique basis $\{f_1,f_2,f_3\}$ of the $\ZZ$-module $S_2(\Gamma,\ZZ)$ satisfying:
\begin{align*}
f_1  =   q - 3q^8 + 4q^{22}  + \ldots, \quad
f_2 =    q^2 - 3q^9 - q^{16}  + \ldots, \quad
f_3 =    q^4 - 4q^{11} + 3q^{18}  + \ldots.
\end{align*}
We have $f_j | \left(\begin{smallmatrix}0 & -1 \\1 & 0 \end{smallmatrix}\right)  = \sum_{k=1}^3  W_{j,k} \cdot f_k$ for a unique matrix $W\in \GL_3(\CC)$.   Using the algorithm of \S\ref{SS:the algorithm}, we find that
\[
W=7\cdot \left(\begin{array}{ccc} -3\xi^2 - 2\xi + 2   &  2\xi^2 - \xi - 6  &  -\xi^2 - 3\xi + 3 \\ 2\xi^2 - \xi - 6  &  \xi^2 + 3\xi - 3  & 3\xi^2 + 2\xi - 2 \\ -\xi^2 - 3\xi + 3  &  3\xi^2 + 2\xi - 2 &  2\xi^2 - \xi - 6 \end{array}\right),
\]
where $\xi := \zeta_{7}+\zeta_{7}^{-1}$.

Now consider the modular curve $X(7):=X_G$ over $\QQ$, where $G$ is the subgroup of $\GL_2(\ZZ/7\ZZ)$ consisting of matrices of the form $\pm \left(\begin{smallmatrix}1 & 0 \\0 & *\end{smallmatrix}\right)$.  For $1\leq j \leq 3$, let $h_j$ be the cusp form in $S_2(\Gamma(7),\ZZ)$ with the same $q$-expansion as $f_j$ except $q$ is replaced by $q_7$.    From the discussion in \S\ref{SS:SL2 action}, we find that $h_1,h_2,h_3$ is a basis of the $\QQ$-vector space $S_2(\Gamma(7),\QQ)=S_2(\Gamma(7),\QQ)^G$.    Applying the methods of \S\ref{S:modular curves}, we find that $X(7)$ has genus $3$ and $F(h_1,h_2,h_3)=0$, where $F(x,y,z)=x^3z-xy^3+yz^3$.    In particular, we deduce that the curve $X(7)$ is isomorphic to the curve in $\PP^2_\QQ$ defined by the equation $x^3z-xy^3+yz^3=0$ (which up to changing the sign of $z$ is the Klein quartic).  For more on the curve $X(7)$, see \cite{MR1722413}.
\end{example}

\begin{example}
Let $G$ be the subgroup of $\GL_2(\ZZ/13\ZZ)$ generated by $\left(\begin{smallmatrix}2 & 0 \\0 & 2\end{smallmatrix}\right)$, $\left(\begin{smallmatrix}1 & 0 \\0 & 5\end{smallmatrix}\right)$, $\left(\begin{smallmatrix}0 & -1 \\1 & 0\end{smallmatrix}\right)$ and $\left(\begin{smallmatrix}1 & 1 \\-1 & 1\end{smallmatrix}\right)$.    The group $G$ contains the scalar matrices in $\GL_2(\ZZ/13\ZZ)$ and its image in $\operatorname{PGL}_2(\ZZ/13\ZZ)$ is isomorphic to the symmetric group $S_4$; these properties uniquely characterize $G$ up to conjugation in $\GL_2(\ZZ/13\ZZ)$.  A model for $X_G$ was first computed by Banwait and Cremona in \cite{MR3263141}.

Set $\zeta:=\zeta_{13}$.  The $\QQ(\zeta)$-vector space $S_2(\Gamma(N),\QQ(\zeta))$ has dimension $50$.   Using the algorithm from \S\ref{SS:the algorithm} and \S\ref{SS:SL2 action}, we can find a basis of this space as well as the natural $\SL_2(\ZZ)$-action.   A computation then shows that the $\QQ$-vector space $S_2(\Gamma(N),\QQ(\zeta))^G$ has dimension $3$ and there is a basis $f_1,f_2,f_3$ characterized by the following $q$-expansions:
{\small
\begin{align*}    
f_1 = & q_{13} + ( \zeta^{11} +  \zeta^{10} +  \zeta^3 +  \zeta^2) q_{13}^2 + (- \zeta^{11} -  \zeta^{10} +
         \zeta^9 +  \zeta^7 +  \zeta^6 +  \zeta^4 -  \zeta^3 -  \zeta^2 + 2) q_{13}^3 +
        \cdots,\\
f_2 = & ( \zeta^{11} +  \zeta^{10} +  \zeta^3 +  \zeta^2) q_{13} + (-4  \zeta^{11} - 4  \zeta^{10} -
         \zeta^9 -  \zeta^7 -  \zeta^6 -  \zeta^4 - 4  \zeta^3 - 4  \zeta^2 - 3) q_{13}^2 \\
         &+  (3  \zeta^{11} + 3  \zeta^{10} - 3  \zeta^9 - 3  \zeta^7 - 3  \zeta^6 - 3  \zeta^4
        + 3  \zeta^3 + 3  \zeta^2 - 5) q_{13}^3 + \cdots, \\
f_3 = &    ( \zeta^9 +  \zeta^7 +  \zeta^6 +  \zeta^4) q_{13} + (4  \zeta^{11} + 4  \zeta^{10} +
        2  \zeta^9 + 2  \zeta^7 + 2  \zeta^6 + 2  \zeta^4 + 4  \zeta^3 + 4  \zeta^2 +
        5) q_{13}^2 \\
        &+ (- \zeta^{11} -  \zeta^{10} + 2  \zeta^9 + 2  \zeta^7 + 2  \zeta^6 +
        2  \zeta^4 -  \zeta^3 -  \zeta^2 + 4) q_{13}^3 + \cdots.
\end{align*}
 }  
Moreover, we have chosen $f_1,f_2,f_3$ so that it is a basis of the $\ZZ$-module $S_2(\Gamma(N),\QQ(\zeta))^G \cap S_2(\Gamma(N), \ZZ[\zeta])$.  Applying the methods of \S\ref{S:modular curves}, we find that $F(f_1,f_2,f_3)=0$, where $F(x,y,z)$ is the polynomial
{
\begin{align*}
&13 x^4 + 13 x^3 y + 25 x^3 z - 8 x^2 y^2 + 9 x^2 y z + 3 x^2 z^2 - 20 x y^3 \\ &- 39 x y^2 z - 34 x y z^2 - 12 x z^3 - 6 y^4 - 15 y^3 z - 14 y^2 z^2 - 5 y z^3.
\end{align*}
}  
We deduce that $X_G$ is isomorphic to the curve in $\PP^2_\QQ$ defined by the equation $F(x,y,z)=0$.



\end{example}        

We have also used our methods to verify models of various modular curves arising from non-split Cartans that occur in the literature; see \cites{MR3253304,MS,MR3867436}.    One benefit of our approach is that it works for general $G$ and does not require any special representation theory.

\subsection{Some related results} \label{SS:related}

There is an alternate method using Eisenstein series to compute the $\SL_2(\ZZ)$-action on the full space $M_k(\Gamma(N))$ where $k\geq 2$; this was not known to the author until after the first draft of this paper was completed.  For simplicity, assume that $N\geq 3$.  In \cite{BN2019}, Brunault and Neururer considered the $\CC$-subalgebra $\calR_N$ of $M_*(\Gamma(N)):= \bigoplus_{k\geq 1} M_k(\Gamma(N))$ generated by certain Eisenstein series $E_{a,b}^{(1)}$ with $a,b\in \ZZ/N\ZZ$.  Citing work of Khuri--Makdisi, they observe that the $k$-th graded part of $\calR_N$ agrees with $M_k(\Gamma(N))$ for all $k\geq 2$.  The $q$-expansion of the $E_{a,b}^{(1)}$ lie in $\QQ(\zeta_N)$ and the right action of $\SL_2(\ZZ)$ on them is explicit.  So for $k\geq 2$, we obtain a basis of $M_k(\Gamma(N),\QQ(\zeta_N))$ along with the action of $\SL_2(\ZZ)$ with respect to this basis.
  
It would be interesting to compare the efficiency of our algorithm versus an Eisenstein series approach to computing the action of $\SL_2(\ZZ)$ on $M_k(\Gamma(N))$.  In \S2.3 of \cite{MR3889557}, Cohen (who is using Eisenstein series to computing the $q$-expansion of a modular form at all cusps) notes that for large $N$, one needs to work numerically in $\CC$ and then if desired use LLL-type algorithms to recongnize the coefficients.   So a reasonable approach would be to use Eisenstein series to do numerical approximations and then the methods of this paper to determine coefficients precisely.    Since the algorithm of \S\ref{SS:the algorithm} requires only recognizing rational integers, it should not need as much accuracy as an LLL-type algorithm.  We will study the Eisenstein series approach in future work.

Collins and Cohen \cites{Collins,MR3889557} have both recently described how to numerically compute the $q$-expansion of a modular form at all of its cusps (both of these papers are interested in numerically computing Petersson inner products).  In private communications, David Loeffler has observed that there is a purely algebraic approach to computing Atkin--Lehner operators using modular symbols.

\subsection{Acknowledgements}  
Thanks to Jeremy Rouse for corrections to an earlier version.
  
 \section{Arithmetic of the Atkin--Lehner operator}
\label{S:arithmetic} 
 
 Fix positive integers $k$ and $N$.  In this section, we prove some arithmetic facts about the action of the Atkin--Lehner operator $W_N$ and the diamond operators on the space of modular forms $M_k(\Gamma_1(N))$.
 
 For each automorphisms $\sigma$ of $\CC$, we have $\sigma(\zeta_N)=\zeta_N^{\chi_N(\sigma)}$ for a unique $\chi_N(\sigma)\in (\ZZ/N\ZZ)^\times$.
 
\begin{thm}  \label{T:AL arithmetic}
Take any congruence subgroup $\Gamma_1(N) \subseteq \Gamma \subseteq \Gamma_0(N)$.  
\begin{romanenum}
\item \label{T:AL arithmetic i}
Let $B$ be a $\ZZ[1/N,\zeta_N]$-subalgebra of $\CC$.     The map $M_k(\Gamma,B) \to M_k(\Gamma,B)$, $f\mapsto f|W_N$ is an isomorphism of $B$-modules.

\item \label{T:AL arithmetic ii}
Let $B$ be a subring of $\CC$. For any $d\in (\ZZ/N\ZZ)^\times$, the map $M_k(\Gamma,B) \to M_k(\Gamma,B)$, $f\mapsto f|\ang{d}$ is an isomorphism of $B$-modules.
\item \label{T:AL arithmetic iii}
For any modular form $f\in M_k(\Gamma_1(N))$ and automorphism $\sigma$ of the field $\CC$, we have
\[
\sigma(f|W_N) =  (\sigma(f)|W_N)|\ang{\chi_N(\sigma)},
\]
\end{romanenum}
\end{thm}
\begin{proof}

We first prove (\ref{T:AL arithmetic i}).  We will make use of Katz's algebraic theory of modular forms, cf.~Chapter II of \cite{MR506271}.  Almost everything we will require is summarized in \S3.6 of \cite{MR1332907}.   Fix a $\ZZ[1/N,\zeta_N]$-subalgebra $B$ of $\CC$.   With definitions as in \S2.1 of \cite{MR506271}, let $R^k(B,\Gamma_{00}(N)^{\operatorname{arith}})$ and $R^k(B,\Gamma_{00}(N)^{\operatorname{naive}})$ be the $B$-modules consisting of $\Gamma_{00}(N)^{\operatorname{arith}}$ and $\Gamma_{00}(N)^{\operatorname{naive}}$ modular forms, respectively, of weight $k$ defined over $B$.

We claim that $f|W_N \in M_k(\Gamma_1(N),B)$ for any modular form $f\in M_k(\Gamma_1(N),B)$.   Fix a modular form $f\in M_k(\Gamma_1(N),B)$.  Since $B\subseteq \CC$, associated to $f$ is a modular form $F\in R^k(B,\Gamma_{00}(N)^{\operatorname{arith}})$, cf.~\S2.4 of \cite{MR506271}.   The two modular forms $f$ and $F$ each have the notion of a $q$-expansion and they agree with each other.    Using the isomorphisms in (2.3.6) of \cite{MR506271}, we obtain from $F$ a modular form  $G\in R^k(B,\Gamma_{00}(N)^{\operatorname{naive}})$.   Since $B$ is a $\ZZ[1/N,\zeta_N]$-algebra, we have a unique isomorphism $\ZZ/N\ZZ \xrightarrow{\sim} \mu_N$ of group schemes over $\Spec B$ satisfying $1\mapsto \zeta_N$.   Using this isomorphism $\ZZ/N\ZZ \cong \mu_N$, we can view $G$ as a modular form in $R^k(B,\Gamma_{00}(N)^{\operatorname{arith}})$.   Associated to $G$, there is a classical weakly modular form $g$ on $\Gamma_1(N)$; \emph{weakly} meaning that it is meromorphic, and not necessarily holomorphic, at the cusps.   A straightforward computation shows that $f|W_N=g$; for example, see Lemma~3.6.5 of \cite{MR1332907} (note that $N^{-1}\cdot f|\tau$ in the notation of \cite{MR1332907} agrees with our $N^{-k} \cdot f|W_N$).  Since $G$ is defined over $B$, the $q$-expansion of $G$, and hence also of $g=f|W_N$, has coefficients in $B$.   This completes the proof of the claim.

The above claim shows that the function $\alpha\colon M_k(\Gamma,B) \to M_k(\Gamma,B)$, $f\mapsto f|W_N$ is well-defined; it is clearly a homomorphism of $B$-modules.   For any $f\in M_k(\Gamma,B)$, we have $\alpha(\alpha(f))=(f|W_N)|W_N = (-N)^k \cdot f$.  Since $N \in B^\times$, this proves that $\alpha$ is an isomorphism of $B$-modules.  This completes the proof of part (\ref{T:AL arithmetic i}).\\

Now fix a subring $B$ of $\CC$ and take any $d\in (\ZZ/N\ZZ)^\times$.    Take any modular form $f\in M_k(\Gamma,B)$.     We have $f|\ang{d} \in M_k(\Gamma)$.   As above, associated to $f$ is a modular form $F\in R^k(B,\Gamma_{00}(N)^{\operatorname{arith}})$.  
With notation as in Chapter II of \cite{MR1332907}, there is a modular form $F'\in R^k(B,\Gamma_{00}(N)^{\operatorname{arith}})$ that satisfies $F'(E,\omega,i)=F(E,\omega,di)$ for all $\Gamma_{00}(N)^{\operatorname{arith}}$-test objects $(E,\omega,i)$.  There is a classical weakly modular form $f'$ on $\Gamma_1(N)$ that corresponds to $F'$ and the coefficients of its $q$-expansions all lie in $B$.   It is straightforward to show that $f'=f|\ang{d}$; this is equation (3.6.7) of \cite{MR1332907}.  So the $q$-expansion of $f|\ang{d}$ has coefficients in $B$ and thus $f|\ang{d} \in M_k(\Gamma,B)$.   Therefore, the map
\[
\alpha_d \colon M_k(\Gamma,B) \to M_k(\Gamma,B), \quad f\mapsto f|\ang{d}
\]
is well-defined; it is clearly a homomorphism of $B$-modules.  The map $\alpha_d$ is invertible and  its inverse is $\alpha_{d^{-1}}$.\\

Part (\ref{T:AL arithmetic iii}) follows from Lemma~3.5.2 of \cite{MR1332907}.  Note that this lemma is stated for cusp forms, but this assumption is not used in the proof.  The lemma is also stated for modular forms with coefficients in an algebraic closure $\Qbar$ of $\QQ$, say in $\CC$; this is not a problem since  $M_k(\Gamma_1(N))$ has a basis consisting of modular forms with algebraic Fourier coefficients.
\end{proof}

\begin{cor}  \label{C:algebraic WN Gamma0}
Fix a number field $K\subseteq \CC$ and a modular form $f\in M_k(\Gamma_0(N),K)$.  Then $f|W_N$ also lies in $M_k(\Gamma_0(N),K)$.
\end{cor}
\begin{proof}
We have $f|W_N$ in $M_k(\Gamma_0(N))$.   Take any automorphism $\sigma$ of $\CC$ that fixes $K$.  We have $\sigma(f)=f$ since $f$ has coefficients in $K$.   By Theorem~\ref{T:AL arithmetic}(\ref{T:AL arithmetic iii}), we have  $\sigma(f|W_N)= (f|W_N)|\ang{\chi_N(\sigma)}=f|W_N$.  Since $\sigma$ was an arbitrary automorphism of $\CC$ that fixes $K$, we deduce that $f|W_N$ has coefficients in $K$.  
\end{proof}

\begin{lemma} \label{L:W and d interaction}
We have $(f|\ang{d})|W_N= (f|W_N)|\ang{d^{-1}}$ for all $f\in M_k(\Gamma_1(N))$ and $d\in  (\ZZ/N\ZZ)^\times$.
\end{lemma}
\begin{proof}
Take any $f\in M_k(\Gamma_1(N))$ and $d\in  (\ZZ/N\ZZ)^\times$. Choose a matrix $\gamma \in \SL_2(\ZZ)$ that is congruent to $\left(\begin{smallmatrix}d & 0 \\0 & d^{-1}\end{smallmatrix}\right)$ modulo $N$.   One can check that $\gamma':=\left(\begin{smallmatrix}0 & -1 \\N & 0\end{smallmatrix}\right)   \gamma  \left(\begin{smallmatrix}0 & -1 \\N & 0\end{smallmatrix}\right)^{-1}$  is in $\SL_2(\ZZ)$ and is congruent to $\left(\begin{smallmatrix}d^{-1} & 0 \\0 & d\end{smallmatrix}\right)$ modulo $N$.  Therefore,
\[
(f|W_N)|\ang{d^{-1}}= N^{k/2}\cdot (f|\left(\begin{smallmatrix}0 & -1 \\N & 0\end{smallmatrix}\right))|\gamma = N^{k/2}\cdot  (f| \gamma')|\left(\begin{smallmatrix}0 & -1 \\N & 0\end{smallmatrix}\right) = (f|\ang{d})|W_N. \qedhere
\]
\end{proof}

\section{Integrality of coefficients}

Fix positive integers $k$ and $N$.   For a modular form $f\in M_k(\Gamma_1(N))$ whose Fourier coefficients are algebraic integers, the coefficients of the modular form $f|W_N$ are algebraic but need not be integral.   

 The goal of this section is to show that the coefficients of $f|W_N$ times an explicit positive integer are all algebraic integers.  Define the integers
\[
B_{k,N}:=\prod_{p|N} p^{\lceil k/(p-1) \rceil} \quad \text{ and } \quad  C_{k,N}:=\prod_{p|N} p^{\lfloor k/(p-1) \rfloor},
\]
where $\lceil x \rceil$ and $\lfloor x \rfloor$ are the values of $x$ rounded up and down, respectively, to the nearest integer.  Note that $C_{k,N}$ divides the integer $C_k:=\prod_{p\leq k+1} p^{\lfloor k/(p-1) \rfloor}$ which depends only on $k$.  

\begin{thm} \label{T:integrality}
\begin{romanenum}
\item \label{T:integrality i}
If $f\in M_k(\Gamma_1(N))$ is a modular form whose Fourier coefficients are algebraic integers, then the coefficients of $B_{k,N}\cdot f|W_N$ are also algebraic integers.
\item \label{T:integrality ii}
If $f\in M_k(\Gamma_0(N))$ is a modular form whose Fourier coefficients are algebraic integers, then the coefficients of $C_{k,N}\cdot f|W_N$ are also algebraic integers.
\end{romanenum}
\end{thm}

\subsection{Valuations}

Take any number field $K\subseteq \CC$ that contains $\zeta_N$.  For a non-zero prime ideal $\p$ of $\OO_K$, let $v_\p\colon K^\times \to \QQ$ be the valuation corresponding to $\p$ normalized so that $v_\p(p)=1$, where $p$ is the rational prime divisible by $\p$.  We set $v_\p(0)=+\infty$.   For each modular form $f \in M_k(\Gamma(N),K)$, define 
\[
v_\p(f) = \inf_{n \geq 0} v_\p(a_n(f)),
\]
where $f$ has $q$-expansion $\sum_{n=0}^\infty   a_n(f) q_N^{n}$.  Note that $v_\p(f)\neq -\infty$ since the $q$-expansion of $f$ actually lies in $K\otimes_\ZZ \ZZ\brak{q_N} \subseteq K\brak{q_N}$.    Similarly, we can define $v_\p(f)$ for any power series $f\in K\otimes_\ZZ \ZZ\brak{q_N}$.    For a modular form $f\in M_k(\Gamma_1(N),K)$, we have $f|W_N \in M_k(\Gamma_1(N),K)$ by Theorem~\ref{T:AL arithmetic}(\ref{T:AL arithmetic i}).

\begin{lemma} \label{L:easy valuation}
Let $\p$ be a non-zero prime ideal of $\OO_K$ that does not divide $N$.  Then $v_\p(f|W_N)=v_\p(f)$ for all $f\in M_k(\Gamma_1(N),K)$.
\end{lemma}
\begin{proof}
We claim that $v_\p(f|W_N)\geq v_\p(f)$ holds for any non-zero $f\in M_k(\Gamma_1(N),K)$.   After scaling $f$ by an appropriate non-zero element of $K$, we may assume without loss of generality that $v_\p(f)=0$.  So $f\in M_k(\Gamma_1(N),B)$, where $B$ is the subring of $K$ consisting of $x\in K$ satisfying $v_\p(x)\geq 0$.  Since $\p\nmid N$, we find that $B$ is a $\ZZ[1/N,\zeta_N]$-subalgebra of $\CC$.  By Theorem~\ref{T:AL arithmetic}(\ref{T:AL arithmetic i}), we deduce that $f|W_N$ also has coefficients in $B$ and hence $v_\p(f|W)\geq 0$.   This proves the claim.   

We now prove the lemma.   We may assume that $f$ is non-zero since otherwise the lemma is trivial.  Applying the claim to $f|W_N$ gives $v_\p((f|W_N)|W_N) \geq v_\p(f|W_N)$.   Since $(f|W_N)|W_N = \pm N^k f$ and $\p\nmid N$, this implies that $v_\p(f) \geq v_\p(f|W_N)$.  This proves the lemma since the claim gives the other inequality $v_\p(f|W_N)\geq v_\p(f)$.
\end{proof}

Note that the $\p$-adic valuations of $f$ and $f|W_N$ need not agree for primes $\p$ dividing $N$.   The following theorem bounds the difference between these two valuations.

\begin{thm} \label{T:DR generalization}
Take any prime $p$ that divides $N$ and any prime ideal $\p$ of $\OO_K$ that divides $p$.   Then 
\[
\big| v_\p(f|W_N)-v_\p(f)- k/2\cdot v_p(N) \big| \leq \frac{k}{2} v_p(N) + \frac{k}{p-1}
\]
for any non-zero $f\in M_k(\Gamma_1(N),K)$.
\end{thm}

\begin{remark}
In the special case where $f\in M_k(\Gamma_0(p),\QQ)$, Theorem~3.3 was proved by Deligne and Rapoport, cf.~Proposition~3.20 in Chapter VII of \cite{MR0337993}.  We prove Theorem~\ref{T:DR generalization} by reducing to this special case.
\end{remark}

\subsection{Proof of Theorem~\ref{T:DR generalization}}
We claim that the inequality
\begin{align} \label{E:valuation difference}
v_\p(f|W_N)-v_\p(f)  \geq - \frac{k}{p-1}
\end{align}
holds for all non-zero $f\in M_k(\Gamma_1(N),K)$.  

Assume that the claim holds.   Take any non-zero $f\in M_k(\Gamma_1(N),K)$.  Applying (\ref{E:valuation difference}) to the modular function $f|W_N$ gives $v_\p((f|W_N)|W_N)-v_\p(f|W_N) \geq -\frac{k}{p-1}$.  Since $(f|W_N)|W_N=\pm N^k f$, we deduce that
\begin{align} \label{E:valuation difference 2}
v_\p(f|W_N)- v_\p(f)- k v_p(N) \leq \frac{k}{p-1}.
\end{align}
The theorem follows from the inequalities (\ref{E:valuation difference}) and (\ref{E:valuation difference 2}).\\

We will prove (\ref{E:valuation difference}) by considering various cases.  Take any non-zero $f\in M_k(\Gamma_1(N),K)$.  Note that there is no harm in scaling $f$ by a non-zero element of $K$ since the value $v_\p(f|W_N)-v_\p(f)$ will not change. In particular, we may assume that $v_\p(f)=0$ when desired.  

Different weights will arise in the proof, so we will add subscripts to slash and Atkin--Lehner operators to indicate the weight involved if it is not $k$.      Let $B$ be the subring of $K$ consisting of $x\in K$ satisfying $v_\p(x)\geq 0$.   For later, note that the power series ring $B\brak{q}$ is integrally closed since $B$ is a PID, cf.~\cite{MR1727221}*{Ch.~V \S4 Prop.~14}. \\

\noindent $\bullet$ \textbf{Case 1}:  \emph{Suppose that $N=p$ and that $f \in M_{k}(\Gamma_0(N),\QQ)$.  }

The claim in this case follows from Proposition~3.20 in Chapter VII of \cite{MR0337993}.\\

\noindent $\bullet$ \textbf{Case 2}:  \emph{Suppose that $N=p$ and that $f \in M_{k}(\Gamma_0(N),K)$.  }

After scaling $f$ by an appropriate non-zero element of $K$, we may assume that $f\in M_k(\Gamma_0(N),\OO_K)$ and $v_\p(f)=0$.     We have $M_k(\Gamma_0(N),\ZZ)\otimes_\ZZ \OO_K = M_k(\Gamma_0(N),\OO_K)$, cf.~section B.1.2 in Appendix B of \cite{MR3709060}.  So there are $f_1,\dots, f_d \in M_k(\Gamma_0(N),\ZZ)$ and $c_1,\ldots, c_d \in \OO_K$ such that $f=\sum_{i=1}^d c_i f_i$.  Therefore, $v_\p(f|W_N) \geq \min_i v_\p(f_i|W_N)$.  By Case 1, we have $v_\p(f_i|W_N) \geq v_\p(f_i) -k/(p-1)\geq -k/(p-1)$ for all $1\leq i \leq d$.    We deduce that $v_\p(f|W_N)\geq -k/(p-1)$. This proves the claimed inequality (\ref{E:valuation difference}) in this case.\\

\noindent $\bullet$  \textbf{Case 3}:  \emph{Suppose that $N=p^{r+1}$ for an integer $r\geq 1$ and that $f \in M_{k}(\Gamma_0(N),K)$.  }

After possibly replacing $K$ by a larger number field, we may assume without loss of generality that there is an element $\pi\in K$ satisfying $v_\p(\pi)=k/(p-1)$.  Define $g:= p^{r k/2}\cdot  f| \left(\begin{smallmatrix}p^r & 0 \\ 0 & 1\end{smallmatrix}\right)^{-1}$.  The function $g$ is a modular form on the congruence subgroup
\[
\Gamma :=  \left(\begin{smallmatrix}p^r & 0 \\ 0 & 1\end{smallmatrix}\right) \Gamma_0(p^{r+1})  \left(\begin{smallmatrix}p^r & 0 \\ 0 & 1\end{smallmatrix}\right)^{-1} = \Big\{ \left(\begin{smallmatrix}a & p^rb \\p c & d\end{smallmatrix}\right) : a,b,c,d \in \ZZ \text{ such that } ad-p^{r+1}bc=1 \Big\}.
\]
We have $g\in M_k(\Gamma,B)$ since $g(\tau)=f(\tau/p^r)$ and $v_\p(f)=0$.

  With the matrix $T=\left(\begin{smallmatrix}1 & 1 \\0 & 1\end{smallmatrix}\right)$, define the polynomial
\[
P(x) := \prod_{j=0}^{p^r-1} (x -  g| T^j).
\]
The $q$-expansion of $g|T^j$ has coefficients in $B$; they are obtained by scaling the coefficients of $g$ by suitable $N$-th roots of unity.   Therefore, $P(x)=\sum_{i=0}^{p^r}  b_i \cdot x^{p^r-i}$ with $b_i \in B\brak{q}$.
Since the matrices $\{T^j : 0\leq j\leq p^r-1\}$ represent the right cosets of $\Gamma$ in $\Gamma_0(p)$, we find that $b_i$ is a modular form for $\Gamma_0(p)$ of weight $ki$ and hence $b_i \in M_{ki}(\Gamma_0(p),B)$.  
By Case 2 applied to $b_i$, we have
\[
v_\p(b_i|_{ki} W_p) \geq - ki/(p-1)
\]
and hence $v_\p(\pi^i\cdot b_i|_{ki} W_p)\geq 0$.  Therefore, $\pi^i\cdot b_i|_{ki}W_p$ is an element of $B\brak{q}$.   

Now define the polynomial
\[
Q(x):=\prod_{j=0}^{p^r-1} (x - \pi \cdot p^{k/2} (g| T^j) |\left(\begin{smallmatrix}0 & -1 \\ p & 0\end{smallmatrix}\right))  = \sum_{i=0}^{p^r} \pi^i \cdot p^{ki/2} b_i|_{ki} \left(\begin{smallmatrix}0 & -1 \\ p & 0\end{smallmatrix}\right) \cdot x^{p^r-i}= \sum_{i=0}^{p^r} \pi^i  \cdot b_i|_{ki}W_p \cdot x^{p^r-i};
\]
it is a monic polynomial with coefficients in $B\brak{q}$.   We have
\begin{align*}
f|W_N &= N^{k/2} f|\left(\begin{smallmatrix}0 & -1 \\ N & 0\end{smallmatrix}\right) = N^{k/2} \cdot p^{-rk/2}\, g| \big( \left(\begin{smallmatrix}p^r & 0 \\ 0 & 1\end{smallmatrix}\right) \left(\begin{smallmatrix}0 & -1 \\ N & 0\end{smallmatrix}\right) \big) = p^{k/2} g| \left(\begin{smallmatrix}0 & -p^r \\ N & 0\end{smallmatrix}\right)  = p^{k/2} g| \left(\begin{smallmatrix}0 & -1 \\ p & 0\end{smallmatrix}\right)
\end{align*}  
and hence $\pi\cdot f|W_N$ is a root of $Q(x)$.      Since $B\brak{q}$ is integrally closed and since $\pi\cdot f|W_N$ is a root of $Q(x) \in (B\brak{q})[x]$ that lies in the fraction field of $B\brak{q}$, we deduce that $\pi\cdot f|W_N$ lies in $B\brak{q}$.  Therefore, $v_\p(\pi\cdot f|W_N)\geq 0$ and hence $v_\p(f|W_N) \geq - v_\p(\pi) = - k/(p-1)$. This proves the claimed inequality (\ref{E:valuation difference}) in this case.\\

\noindent $\bullet$ \textbf{Case 4}: \emph{Suppose that $f \in M_{k}(\Gamma_0(N),K)$.  }

After scaling $f$ by a non-zero element of $K$, we may assume that $v_\p(f)=0$.  Let $p^r$ be the largest power of $p$ that divides $N$.  Set $M:=N/p^r$.      Let $R \subseteq \SL_2(\ZZ)$ be a set of representatives of the right cosets of $\Gamma_0(M)$ in $\SL_2(\ZZ)$ chosen so that each $A\in R$ is congruent to the identity matrix modulo $p^r$.   Define
\[
g := \prod_{A\in R} f|A.
\]
The function $g$ is an element of $M_{km}(\Gamma_0(p^r),K)$ with $m:=|R|$.   We have $v_\p(f|A)=v_\p(f)$ for all $A\in R$ by \cite{MR0337993}*{VII~Corollaire~3.12} and our assumption that all $A\in R$ are congruent modulo $p^r$ to the identity matrix.   Therefore,  $v_\p(g)=\sum_{A\in R} v_\p(f|A) = m v_\p(f)$.    We have $v_\p(g)=m\cdot 0= 0$ since $v_\p(f)=0$, so $g$ is an element of $M_{k'}(\Gamma_0(p^r),B)$ with $k':=km$.   By Case 2 or Case 3 applied to $g$, we have
\begin{align} \label{E:pr version}
v_\p(g|_{k'} W_{p^r})  \geq - km/(p-1).
\end{align}

Define the matrix $S=\left(\begin{smallmatrix}0 & -1 \\1 & 0\end{smallmatrix}\right)$.  Take any $A\in R$.  Since $S^{-1}AS \in \SL_2(\ZZ)$ is congruent to the identity matrix modulo $p^r$, we have
\[
v_\p(f|S) = v_\p((f|S)|(S^{-1}AS))=v_\p( (f|A)|S )
\]
by \cite{MR0337993}*{VII~Corollaire~3.12}.  Therefore, $v_\p(g|_{k'} S)=\sum_{A\in R} v_\p((f|A)|S)= m\cdot v_\p(f|S)$.   We have
\[
f|W_N = N^{k/2} f | \left(\begin{smallmatrix}0 & -1 \\N & 0\end{smallmatrix}\right)
= N^{k/2} (f|S)| \left(S^{-1} \left(\begin{smallmatrix}0 & -1 \\N & 0\end{smallmatrix}\right)\right)= N^{k/2} (f|S)| \left(\begin{smallmatrix}N & 0 \\0 & 1\end{smallmatrix}\right).
\]
So $(f|W_N)(\tau)= N^k (f|S)(N\tau)$ and hence $v_\p(f|W_N)=k\cdot v_\p(N) + v_\p(f|S)$.   Similarly, $v_\p(g|_{k'} W_{p^r})= k'\cdot v_\p(p^r)+v_\p(g|_{k'} S)$.   Therefore, 
\begin{align*}
m \cdot v_\p( f|W_N)  = km v_\p(N) + m v_\p(f|S) =  k' v_\p(p^r) +  v_\p(g|_{k'} S) = v_\p(g|_{k'}W_{p^r}).
\end{align*}
By (\ref{E:pr version}), we deduce that $v_\p( f|W_N) \geq -k/(p-1)$.  This proves the claimed inequality (\ref{E:valuation difference}) in this case.\\

\noindent $\bullet$ \textbf{Case 5}:  \emph{General case.}

We may assume that $v_\p(f)=0$.  Take any $d\in (\ZZ/N\ZZ)^\times$.   By Theorem~\ref{T:AL arithmetic}(\ref{T:AL arithmetic ii}), the diamond operator $\ang{d}$ acts as an automorphism on $M_k(\Gamma_1(N),B)$.  This implies that $v_\p(h|\ang{d})=v_\p(h)$ for all $h\in M_k(\Gamma_1(N),K)$.  

Define
\[
g := \prod_{d\in (\ZZ/N\ZZ)^\times} f|\ang{d};
\]
it is a modular form on $\Gamma_0(N)$ of weight $k':=k \varphi(N)$.    Therefore, $v_\p(g)= \sum_{d} v_\p(f|\ang{d}) = \varphi(N) v_\p(f)=0$.  We have
\[
g|_{k'} W_N = \prod_{d\in (\ZZ/N\ZZ)^\times} (f|\ang{d})|W_N= \prod_{d\in (\ZZ/N\ZZ)^\times} (f|W_N)|\ang{d^{-1}},
\]
where the last equality uses Lemma~\ref{L:W and d interaction}.  Therefore, $v_\p(g|_{k'} W_N)=\sum_d v_\p((f|W_N)|\ang{d^{-1}}) = \varphi(N) v_\p(f|W_N)$ and hence
\[
v_\p(f|W_N) = \varphi(N)^{-1} \cdot v_\p(g|_{k'} W_N) \geq - \varphi(N)^{-1} \cdot k'/(p-1) = -k/(p-1),
\]
where the inequality uses Case 4 applied to $g$.   This completes the proof of the claimed inequality (\ref{E:valuation difference}).

\subsection{Proof of Theorem~\ref{T:integrality}} \label{SS:proof integrality}
Take any $f\in M_k(\Gamma_1(N))$ whose Fourier coefficients are algebraic integers.   There is  a number field $K \subseteq \CC$ that contains the coefficients of $f$ and the $N$-th root of unity $\zeta_N$.    We thus have $f\in M_k(\Gamma_1(N),\OO_K)$.  We have $f|W_N \in M_k(\Gamma_1(N),K)$ by Theorem~\ref{T:AL arithmetic}(\ref{T:AL arithmetic i}).   So to prove that $B_{k,N}\cdot f|W_N$ has coefficients in $\OO_K$, it suffices to show that $v_\p(B_{k,N}\cdot f|W_N) \geq 0$ for all non-zero primes $\p$ of $\OO_K$.   Take any non-zero prime ideal $\p\subseteq \OO_K$.   If $\p\nmid N$, then 
\[
v_\p(B_{k,N} \cdot f|W_N)=v_\p(f|W_N) = v_\p(f) \geq 0,
\] 
where we have used  Lemma~\ref{L:easy valuation} and that $f$ has coefficients in $\OO_K$.   

Now suppose that $\p$ divides $N$.  We have $v_\p(B_{k,N})=\lceil k/(p-1) \rceil$.  By Theorem~\ref{T:DR generalization} with $v_\p(f)\geq 0$, we have  $v_\p(f|W_N) \geq -k/(p-1)$.  Therefore, $v_\p(B_{k,N}\cdot f|W_N)  \geq \lceil k/(p-1) \rceil - k/(p-1) \geq 0$.  This completes the proof of part (\ref{T:integrality i}).\\

We now prove (\ref{T:integrality ii}).  Take any $f\in M_k(\Gamma_0(N))$ whose Fourier coefficients are algebraic integers.      Choose a number field $K \subseteq \CC$ for which $f\in M_k(\Gamma_0(N),\OO_K)$.    Without loss of generality, we may assume that $f\in M_k(\Gamma_0(N),\ZZ)$ since  $M_k(\Gamma_0(N),\OO_K)=M_k(\Gamma_0(N),\ZZ) \otimes_\ZZ \OO_K$, cf.~section B.1.2 in Appendix B of \cite{MR3709060}.   We have $f|W_N \in M_k(\Gamma_0(N),\QQ)$ by Corollary~\ref{C:algebraic WN Gamma0}.

Take any prime $p$.  If $p \nmid N$, then $v_p(C_{k,N}\cdot f|W_N) = v_p(f|W_N)=v_p(f) \geq 0$, where we have used  Lemma~\ref{L:easy valuation} and that $f$ has coefficients in $\ZZ$.  Now suppose that $p$ divides $N$.   By Theorem~\ref{T:DR generalization} and $v_p(f)\geq 0$, we have $v_p(f|W_N)\geq -k/(p-1)$.  Therefore,
\[
v_p(C_{k,N} \cdot f|W_N)=\lfloor k/(p-1) \rfloor + v_p(f|W_N) \geq \lfloor k/(p-1) \rfloor - k/(p-1) > -1.
\]
The coefficients of $C_{k,N}\cdot f|W_N$ lie in $\QQ$ so $v_p(C_{k,N}\cdot f|W_N)$ is an integer.   Therefore, $v_p(C_{k,N} \cdot f|W_N)\geq 0$ since $v_p(C_{k,N} \cdot f|W_N)$ is an integer strictly larger than $-1$.   We have $C_{f,N}\cdot f|W_N \in M_k(\Gamma_0(N),\ZZ)$ since its coefficients are rational and have non-negative valuation at all non-zero primes $p$. This proves (\ref{T:integrality ii}).

\section{Spaces of modular forms} \label{S:numerical}

Fix positive integers $k$ and $N$.   In this section, we verify that several subspaces $\calM$ of $M_k(\Gamma_1(N))$ satisfy  conditions (\ref{c:a})--(\ref{c:d}) of \S\ref{SS:setup intro}. In \S\ref{SS:numerical}, we explain how to numerically approximate the action of $W_N$ and diamond operators on $M_k(\Gamma_1(N))$.  

\subsection{Generators}

Recall that for a subspace $\calM\subseteq M_k(\Gamma_1(N))$ and a subring $R\subseteq \CC$, we defined $\calM(R)$ to be $\calM \cap M_k(\Gamma_1(N),R)$.

\begin{lemma} \label{L:finite index SNF}
Let $\{h_1,\ldots, h_r\}$ be a set that generates a finite index subgroup of $\calM(\ZZ)$.  Assume further that one can compute an arbitrary number of terms in the $q$-expansion of each $h_i$.     Then one can find a basis $f_1,\ldots, f_g$ of the $\ZZ$-module $\calM(\ZZ)$ for which one can compute an arbitrary number of terms in the $q$-expansion of each $f_i$.  
\end{lemma}
\begin{proof}
Let $L$ be the subgroup of $\calM(\ZZ)$ generated by $\{h_1,\ldots, h_r\}$.    Let $s$ be the largest integer for which $s\leq k/12\cdot [\SL_2(\ZZ):\Gamma_1(N)]$.  For each $f\in \calM$, we have a $q$-expansion $\sum_{i=0}^\infty a_i(f) q^i$.   For a prime $p$, Sturm's bound (Theorem~9.18 of \cite{MR2289048})  says that if $f\in \calM(\ZZ)$ satisfies $a_i(f) \equiv 0 \pmod{p}$ for all $i\leq s$, then $a_i(f)\equiv 0\pmod{p}$ for all $i$.    In particular, if $f\in \calM(\ZZ)$ satisfies $a_i(f) =0$ for all $i\leq s$, then $a_i(f)=0$ for all $i$.

Let $g$ be the rank of the $\ZZ$-module $\calM(\ZZ)$.  So by replacing $\{h_i\}$ by a suitable subset, we may assume that $r=g$ and that the modular forms $h_1,\ldots, h_g$ are a basis for $L$ (Sturm's bound ensures that we can check linear independence by only considering a finite number of terms of the $q$-expansions).\\

Let $A$ be the $g \times s$ matrix satisfying $A_{i,j}=a_{j-1}(h_i)$.  Since the $h_1,\ldots, h_g$ are linearly independent in $\calM(\ZZ)$, Sturm's bound implies that $A\in M_{g,s}(\ZZ)$ has rank $g$.     Recall that there are unique positive integers $b_1,\ldots, b_g$ satisfying $b_i | b_{i+1}$ for all $1\leq i <g$ such that there are matrices $U\in \GL_g(\ZZ)$ and $V\in \GL_s(\ZZ)$ with $(U A V)_{i,j}$ equal to $b_i$ when $i=j$ and $0$ otherwise.   This is the \emph{Smith normal form} of $A$ and is straightforward to compute.

First suppose that $b_1\neq 1$.  Choose a prime $p$ dividing $b_1$.   Then $A$ modulo $p$ has rank strictly less than $g$.   So there are $c_1,\ldots, c_g \in \ZZ$, not all divisible by $p$, such that $c_1 a_i(h_1)+ \cdots +  c_g a_i(h_g) \equiv 0 \pmod{p}$ for all $i\leq s$.   By Sturm's bound, all the coefficients of the $q$-expansion of $c_1 a_i(h_1)+ \cdots +  c_g a_i(h_g)$ are divisible by $p$.   So $f := c_1/p\cdot h_1+ \cdots +  c_g/p\cdot  h_g$ is an element of $\calM(\ZZ)$.  We have $f\notin L$ since not all of the $c_i$ are divisible by $p$.   Let $L'$ be the subgroup of $\calM(\ZZ)$ generated by $h_1,\ldots, h_g$ and $f$.   By replacing $L$ by $L'$ and choosing a new basis $h_1,\ldots, h_g$ of $L'$, we can repeat this process until $b_1 = 1$.

Now suppose that $b_1 = 1$.  We claim that $L= \calM(\ZZ)$ and hence $h_1,\ldots, h_g$ is a basis of $\calM(\ZZ)$.  Suppose on the contrary that $L\neq \calM(\ZZ)$ and hence there is a prime $p$ and a modular form $f\in \calM(\ZZ)$ such that $pf \in L$ and $f\notin L$.  We have $p f = c_1h_1+\cdots c_g h_g$ for unique $c_i\in \ZZ$.   If not all the $c_i$ are divisible by $p$, then we find that $A$ modulo $p$ has rank strictly less than $g$.  However, $A$ modulo $p$ has rank $g$ since $p\nmid b_1=1$.     This contradicts proves the claim.

Once we have found a basis of $\calM(\ZZ)$, the parts of the lemma concerning arbitrarily many terms is immediate.
\end{proof}

\begin{lemma} \label{L:M reduction}
Assume that $\calM = \bigoplus_{i=1}^m \calM_i \subseteq M_k(\Gamma_1(N))$, where each $\calM_i$ is a subspace of $M_k(\Gamma_1(N))$ satisfying conditions (\ref{c:a})--(\ref{c:d}) of \S\ref{SS:setup intro}.   Then $\calM$ also satisfies conditions (\ref{c:a})--(\ref{c:d}).
\end{lemma}
\begin{proof}
Define $L:= \oplus_{i=1}^m \calM_i(\ZZ) \subseteq \calM(\ZZ)$.  Conditions (\ref{c:a})--(\ref{c:c}) for $\calM$ are immediate consequences of those from the $\calM_i$.    Since $L$ spans $\calM$, we find that $L$ is a finite index subgroup of $\calM(\ZZ)$.    Condition (\ref{c:d}) for $\calM$ follows from condition (\ref{c:d}) of the $\calM_i$ and Lemma~\ref{L:finite index SNF}.
\end{proof}

Let $\calM$ be a subspace of $M_k(\Gamma_1(N))$ that is stable under the $\Aut(\CC)$-action.   For a number field $L\subseteq \CC$ and modular form $f\in \calM(L)$, we define 
\[
\Tr_{L/\QQ}(f) = \sum_{\sigma\colon L \hookrightarrow \CC} \sigma(f),
\] 
where $\sigma$ varies over the field embeddings $L\hookrightarrow \CC$.    We have $\Tr_{L/\QQ}(f) \in \calM$ since $\calM$ is stable under the action of $\Aut(\CC)$.    The $q$-expansion of $\Tr_{L/\QQ}(f)$ is obtained from the $q$-expansion of $f$ by taking the trace of the coefficients and hence $\Tr_{L/\QQ}(f) \in \calM(\QQ)$.   If $f\in \calM(\OO_L)$, then $\Tr_{L/\QQ}(f) \in \calM(\ZZ)$.

The following lemma is useful for finding a set of modular forms that generate a finite index subgroup of $\calM(\ZZ)$.   Choosing a linear independent subset, one can then use Lemma~\ref{L:finite index SNF} to compute a basis of $\calM(\ZZ)$.

\begin{lemma} \label{L:gen calM}
Let $\calM$ be a subspace of $M_k(\Gamma_1(N))$ that is stable under the $\Aut(\CC)$-action.  Let $\{h_1,\ldots, h_s\}$ be a subset of $\calM$ that is not contained in any proper subspace of $\calM$ stable under the $\Aut(\CC)$-action.  Further assume that  for each $1\leq i \leq s$, we have $h_i\in \calM(\OO_{L_i})$ for a number field $L_i\subseteq \CC$. 

 For each $1\leq i \leq s$, choose an $a_i\in \OO_{L_i}$ such that $L=\QQ[a_i]$.  Then the set 
\begin{align} \label{E:gen set}
\big\{ \Tr_{L_i/\QQ}(a_i^{j-1} \, h_i ) : 1\leq i \leq s,\,  1\leq j \leq [L_i:\QQ] \big\}
\end{align}
spans $\calM$ and generates a finite index subgroup of $\calM(\ZZ)$.  
\end{lemma}
\begin{proof}
Let $S$ be the set (\ref{E:gen set}) and let $W$ be the span of $S$ in $\calM$.    The set $W$ is stable under the action of $\Aut(\CC)$ since $S \subseteq \calM(\ZZ)$.  So to prove that $S$ spans $\calM$, it suffices to show that each $h_e$ is in $W$ for each $1\leq e \leq s$.  

Fix $1\leq e \leq s$.  Set $d=[L_e:\QQ]$ and let $\sigma_1,\ldots, \sigma_d\colon L_e \hookrightarrow \CC$ be the distinct complex embeddings of $L_e$.    For any $1\leq j \leq d$, we have
\begin{align} \label{E:tr dep}
\Tr_{L_e/\QQ}(a_e^{j-1} \, h_e ) = \sum_{i=1}^d \sigma_i(a_e)^{j-1} \cdot \sigma_i(h_e).
\end{align}
The $d\times d$ matrix $B$ with $B_{i,j}=\sigma_i(a_e)^{j-1}$ has determinant $\prod_{1\leq i<j\leq d} (\sigma_j(a_e)-\sigma_i(a_e))\neq 0$.  Therefore from (\ref{E:tr dep}), we find that each $\sigma_i(h_e)$ is in the complex vector space spanned by $\{\Tr_{L_e/\QQ}(a_e^{j-1} \, h_e ): 1\leq j \leq d\}$.  In particular, $h_e$ is an element of $W$.    Since we took any $1\leq e \leq s$, this proves that $S$ spans $W$.
\end{proof}

\subsection{Newforms} \label{SS:newforms}

A \defi{newform} of weight $k$ and level $N$ is a cusp form $f\in S_k(\Gamma_1(N))$ that is an eigenform for all the Hecke operators $T_n$ and satisfies $a_1(f)=1$.   We have $T_n(f)=a_n(f) f$ for all $n\geq 1$.   Let $\calN(k,N)$ be the set of {newforms} of level $N$.     The set $\calN(k,N)\subseteq S_k(\Gamma_1(N))$ is stable under the action of $\Aut(\CC)$.   

Fix a newform $f\in \calN(k,N)$.   The coefficients of the $q$-expansion of $f$ generate a number field $L$ whose degree we will denote by $g$.     One can compute the field $L$ and can also compute an arbitrary number of terms of the $q$-expansion of $f$, see Algorithm 9.14 in \cite{MR2289048} when $k\geq 2$ for an algorithm using modular symbols.   \\

We now briefly discuss the harder excluded $k=1$ case. For an odd character $\varepsilon \colon (\ZZ/N\ZZ)^\times \to \CC^\times$, one can compute a basis of 
\[
M_1(\Gamma_1(N))(\varepsilon):=\{f\in M_1(\Gamma_1(N)) :\, f|\ang{d} = \varepsilon(d) f \, \text{ for all } d\in (\ZZ/N\ZZ)^\times\}
\]  
for which we can determine arbitrarily many terms of their $q$-expansions and all the coefficients lie in a cyclotomic extension; for example, see \cite{MR3384865}.    The action of the Hecke operators $T_n$ on a modular form in $M_1(\Gamma_1(N))(\chi)$ can be computed from its $q$-expansion.     From this, we can compute a basis of $M_1(\Gamma_1(N))= \oplus_\varepsilon M_1(\Gamma_1(N))(\varepsilon)$ for which we can determine arbitrarily many terms of their $q$-expansions (and all the coefficients lie in a cyclotomic extension) and we know the action of the diamond and Hecke operators with respect to this basis.     By simultaneously diagonalizing the action of $T_p$ for several small primes $p$, we can compute the newforms in $\calN(1,N)$ (we can check that a modular form $f\in M_1(\Gamma_1(N))$ is a cusp form by verifying that $f^2\in S_2(\Gamma_1(N))$).

\subsection{Pseudo-eigenvalues}

We now describe how the Atkin--Lehner operator $W_N$ acts on a newform $f\in \calN(k,N)$.
 
 \begin{prop} \label{P:W action on newforms}
For $f\in \calN(k,N)$, we have
\[
f|W_N = \lambda_N(f) \cdot (-1)^k N^{k/2} \cdot \bbar{f},
\]
where $\lambda_N(f)\in \CC$ is an algebraic number with absolute value $1$ and $\bbar{f} \in \calN(k,N)$ is obtained by applying complex conjugation to the coefficients of the $q$-expansion of $f$.
\end{prop}
\begin{proof}
See \S1 and Theorem~1.1 of \cite{MR508986}.  Note that our Atkin--Lehner operators are normalized differently and the proposition is stated so that $\lambda_N(f)$ agrees with the value from \cite{MR508986}.
\end{proof}

The number $\lambda_N(f)$ is called the \defi{pseudo-eigenvalue} of $f$.   In special cases, one has a closed expression for $\lambda_N(f)$.  For example, if $N$ is squarefree, then an expression for $\lambda_N(f)$ in terms of Gauss sums can be found in \cite{MR0427235}.  \\

Let $\varepsilon_f\colon (\ZZ/N\ZZ)^\times \to \CC^\times$ be the \defi{nebentypus} of $f$; it is the unique such Dirichlet character satisfying $f|\ang{d}=\varepsilon_f(d) f$ for all $d\in (\ZZ/N\ZZ)^\times$.   Note that one can determine $\varepsilon_f$ from the $q$-expansion of $f$ since $a_{p^2}(f)=a_p(f)^2-\varepsilon_f(p) p^{k-2}$ for all primes $p\nmid N$.

\subsubsection{Approximating pseudo-eigenvalues} \label{SS:pseudo approx}

For a fixed newform $f\in \calN(k,N)$, we now describe how to numerically approximate $\lambda_N(f) \in \CC$ from enough terms of the $q$-expansion of $f$.   

For a real number $b>0$, substituting $\tau=i\cdot b/N^{1/2}$ into the equation from Proposition~\ref{P:W action on newforms} gives
\[
(i\cdot b/N^{1/2})^{-k} \cdot f(i\cdot b^{-1}/N^{1/2}) = \lambda_N(f)\cdot (-1)^k N^{k/2} \cdot \bbar{f}(i\cdot b/N^{1/2}) 
\]
and hence
\begin{align} \label{E:lambda comp}
 i^k b^{-k} \, \sum_{n=1}^\infty a_n(f) \,(e^{-2\pi /(bN^{1/2})})^n= \lambda_N(f) \, \sum_{n=1}^\infty \bbar{a_n(f)}\, (e^{-2\pi b/N^{1/2}})^n.
\end{align}

If we know the coefficient $a_n(f)$ for all $n\leq  N$, then we can compute the sums $\sum_{n=1}^N a_n(f) \,(e^{-2\pi /(bN^{1/2})})^n$ and $\sum_{n=1}^N \bbar{a_n(f)}\, (e^{-2\pi b/N^{1/2}})^n.$   We can then approximate the series in (\ref{E:lambda comp}); the error terms of these approximations can be bounded using that $|a_n(f)| \leq d(n)n^{k/2}$ by Deligne, where $d(n)$ is the number of divisors of $n$. 

If $\bbar{f}(i\cdot b/N^{1/2})$ is non-zero, then (\ref{E:lambda comp}) gives a formula for $\lambda_N(f)$ that can be used to approximate it by computing enough terms of each series.   Note that we have $\bbar{f}(i\cdot b/N^{1/2})\neq 0$ away from a discrete set of $b>0$ since $\bbar{f}$ is non-zero and holomorphic.      In practice, one wants to choose $b$ close to $1$; this ensure that both series converge absolutely at a similar rate.

\subsection{Atkin--Lehner--Li theory} \label{SS:ALL}
  
For background, see \S9.2 of \cite{MR2289048}.    For positive divisors $M$ of $N$ and $d$ of $N/M$, we have a degeneracy map
\[
\alpha_d\colon S_k(\Gamma_1(M))\hookrightarrow  	S_k(\Gamma_1(N)),\quad f(\tau) \mapsto f(d\tau).
\]
On $q$-expansions, we have $\alpha_d(\sum_{n=1}^\infty a_n q^n ) = \sum_{n=1}^\infty a_n q^{dn}$.

The \defi{old subspace} of $S_k(\Gamma_1(N))$ is the subspace generated by $\alpha_d(S_k(\Gamma_1(M)))$ for all positive divisors $M|N$ with $M\neq N$ and $d|(N/M)$; we denote it by $S_k(\Gamma_1(N))_\old$.   The \defi{new subspace} of $S_k(\Gamma_1(N))$, which we denote by $S_k(\Gamma_1(N))_\new$, is the orthogonal complement of the subspace $S_k(\Gamma_1(N))_\old$ of $S_k(\Gamma_1(N))$ with respect to the Petersson inner product.   The set $\calN(k,N)$ of newforms is a basis of $S_k(\Gamma_1(N))_{\new}$.   We have a decomposition
\begin{align} \label{E:newform decomp}
S_k(\Gamma_1(N)) = \bigoplus_{M|N} \bigoplus_{d|\frac{N}{M}} \alpha_d(S_k(\Gamma_1(M))_{\new}).
\end{align}

\begin{lemma} \label{L:W action on eigenbasis}
Take any positive divisors $M|N$ and $d|(N/M)$, and set $e:=N/(dM)$.  For any $f \in S_k(\Gamma_1(M))_{\new}$ and $m\in (\ZZ/N\ZZ)^\times$, we have
\[
\alpha_d(f)|W_N = e^{k} \, \alpha_e(f|W_M) \quad \text{ and }\quad \alpha_d(f)|\ang{m} = \alpha_d(f|\ang{m}).
\]
\end{lemma}
\begin{proof}
We have
\begin{align*}
\alpha_d(f)|W_N & = d^{-k/2} N^{k/2} (f| \left(\begin{smallmatrix}d & 0 \\ 0 & 1\end{smallmatrix}\right))|\left(\begin{smallmatrix}0 & -1 \\ N & 0\end{smallmatrix}\right)\\
&= d^{-k/2} N^{k/2} M^{-k/2} (f|W_M)|(\left(\begin{smallmatrix}0 & -1 \\ M & 0\end{smallmatrix}\right)^{-1} \left(\begin{smallmatrix}d & 0 \\ 0 & 1\end{smallmatrix}\right)\left(\begin{smallmatrix}0 & -1 \\ N & 0\end{smallmatrix}\right))\\
&= e^{k/2}  (f|W_M)|(\left(\begin{smallmatrix}0 & -1 \\ M & 0\end{smallmatrix}\right)^{-1} \left(\begin{smallmatrix}d & 0 \\ 0 & 1\end{smallmatrix}\right)\left(\begin{smallmatrix}0 & -1 \\ N & 0\end{smallmatrix}\right))\\
&= e^{k/2}  (f|W_M)|\left(\begin{smallmatrix}N/M & 0 \\ 0 & d\end{smallmatrix}\right) = e^{k/2}  (f|W_M)|\left(\begin{smallmatrix}e & 0 \\ 0 & 1\end{smallmatrix}\right) = e^{k}  \alpha_e(f|W_M). 
\end{align*}
Now take any $m \in (\ZZ/N\ZZ)^\times$ and choose an $\gamma\in \SL_2(\ZZ)$ satisfying $\gamma\equiv \left(\begin{smallmatrix}m^{-1} & 0 \\ 0 & m\end{smallmatrix}\right) \pmod{N}$.  We have
\begin{align*}
\alpha_d(f|\ang{m}) &= d^{-k/2} \, f| (\gamma \left(\begin{smallmatrix}d & 0 \\ 0 & 1\end{smallmatrix}\right)) =  \alpha_d(f)| (\left(\begin{smallmatrix}d & 0 \\ 0 & 1\end{smallmatrix}\right)^{-1} \gamma \left(\begin{smallmatrix}d & 0 \\ 0 & 1\end{smallmatrix}\right)) = \alpha_d(f)|\ang{m},
\end{align*}
where the last equality uses that $\left(\begin{smallmatrix}d & 0 \\ 0 & 1\end{smallmatrix}\right)^{-1} \gamma \left(\begin{smallmatrix}d & 0 \\ 0 & 1\end{smallmatrix}\right)$ is in $\SL_2(\ZZ)$ and is congruent to $\left(\begin{smallmatrix}m^{-1} & * \\ 0 & m\end{smallmatrix}\right)$ modulo $N$.
\end{proof}

\subsection{The space $\calM_f$} \label{SS:Mf}

Fix a newform $f\in \calN(k,N)$.   The coefficients of the $q$-expansion of $f$ generate a number field $L$ whose degree we will denote by $g$.  As noted in \S\ref{SS:newforms}, one can compute the field $L$ and arbitrarily many terms of the $q$-expansion of $f$.    Fix an $a\in \OO_L$ satisfying $L=\QQ[a]$. 

  We define $\calM_f$ to be the subspace of $S_k(\Gamma_1(N))$ generated by $\sigma(f)$ with $\sigma\in \Aut(\CC)$.   Moreover, the set $\{\sigma_1(f),\ldots, \sigma_g(f)\}$ is a basis of $\calM_f$, where $\sigma_1,\ldots, \sigma_g\colon L\hookrightarrow \CC$ are the distinct complex embeddings of $L$.      By Lemma~\ref{L:gen calM}, the set 
\begin{align} \label{E:Tr(af) basis}
\big\{ \Tr_{L/\QQ}(a^{j-1} f ) : 1\leq j \leq g  \big\}
\end{align}
spans $\calM_f$ and generates a finite index subgroup of $\calM_f(\ZZ)$.  Moreover, the set (\ref{E:Tr(af) basis}) is a basis of the $g$-dimensional vector space $\calM_f$.   In particular, $\calM_f(\ZZ)$ spans $\calM_f$.   Since we can compute arbitrarily many terms of the $q$-expansion of $f$, we can compute arbitrarily many terms of the $q$-expansion of the modular forms in the set (\ref{E:Tr(af) basis}).   By Lemma~\ref{L:finite index SNF}, we can find a basis $f_1,\ldots, f_g$ of the $\ZZ$-module $\calM_f(\ZZ)$ such that an arbitrary number of terms in the $q$-expansion of each $f_i$ can be computed.   
  
 The space $\calM_f$ is stable under the action of $W_N$ since by Proposition~\ref{P:W action on newforms}, we have 
\[
\sigma(f)|W_N = \lambda_N(\sigma(f)) \cdot (-1)^k N^{k/2} \cdot \bbar{\sigma(f)}
\]
for each $\sigma\in \Aut(\CC)$.    The space $\calM_f$ is stable under the action of $\ang{d}$, with $d\in (\ZZ/N\ZZ)^\times$, since $\sigma(f)|\ang{d} = \varepsilon_{\sigma(f)}(d)\, \sigma(f)$ for each $\sigma\in \Aut(\CC)$.  
 
We have now verified the following.

\begin{lemma} \label{L:Mf conditions}
For each $f\in \calN(k,N)$, $\calM_f$ satisfies the conditions (\ref{c:a})--(\ref{c:d}) from \S\ref{SS:setup intro}.
\end{lemma}

\subsection{Cusp forms} \label{SS:cusp forms}

Since $\calN(k,N)$ is a basis of $S_k(\Gamma_1(N))_{\new}$, we have
\[
S_k(\Gamma_1(N))_{\new} = \bigoplus_{f\in \calN'(k,N)} \calM_f,
\]
where $\calN'(k,N)$ is a set of representatives of the $\Aut(\CC)$-orbits on $\calN(k,N)$.  By (\ref{E:newform decomp}), we have
\begin{align} \label{E:newform decomp 2}
S_k(\Gamma_1(N)) = \bigoplus_{M|N} \bigoplus_{f\in \calN'(k,M)}  \bigoplus_{d|\frac{N}{M}}  \alpha_d(\calM_f) = \bigoplus_{f\in \calN'(k,M)} \bigoplus_{M|N} \bigoplus_{de=N/M, \, d\leq e}  \calM_{f,d},
\end{align}
where for $d|(N/M)$ and $e=(N/M)/d$ we define
\[
\calM_{f,d}:= 
\begin{cases}
\alpha_d(\calM_f)\oplus \alpha_e(\calM_f) & \text{ if $d\neq e$},\\ 
\alpha_d(\calM_f) & \text{ if $d=e$}.
\end{cases}   
\]

\begin{lemma} \label{L:Mfd conditions}
Take any positive divisors $M|N$ and $d|(N/M)$, and set $e:=(N/M)/d$.  Then for any newform $f\in \calN(k,M)$, $\calM_{f,d}$  satisfies conditions (\ref{c:a})--(\ref{c:d}) of \S\ref{SS:setup intro}.
\end{lemma}
\begin{proof}
From Lemma~\ref{L:Mf conditions}, $\calM_f$ satisfies the conditions (\ref{c:a})--(\ref{c:d}) from \S\ref{SS:setup intro} with $N$ replaced by $M$.   There is a basis $f_1,\ldots, f_g$ of the $\ZZ$-module $\calM_f(\ZZ)$ that is a basis of $\calM_f$ and for which arbitrary number of terms of the $q$-expansion of each $f_i$ can be computed.

Since $\calM_f$ is stable under the action of $W_M$ and the diamond operators, Lemma~\ref{L:W action on eigenbasis} implies that $\calM_{f,d}$ is stable under the action of $W_N$ and the diamond operators.  The set $\{\alpha_d(f_1),\ldots, \alpha_d(f_g), \alpha_e(f_1),\ldots, \alpha_e(f_g)\}$ generates a finite index subgroup of $\calM_{f,d}(\ZZ)$ and spans $\calM_{f,d}$.   By Lemma~\ref{L:finite index SNF}, we can find a basis of $\calM_{f,d}(\ZZ)$ for which an arbitrary number of terms of the $q$-expansions can be computed.    We have thus verified that $\calM_{f,d}$  satisfies conditions (\ref{c:a})--(\ref{c:d}).
\end{proof}

Now take any congruence subgroup $\Gamma_1(N) \subseteq \Gamma \subseteq \Gamma_0(N)$.   Let $H$ be the subgroup of $(\ZZ/N\ZZ)^\times$ such that $\Gamma$ consists of the matrices in $\SL_2(\ZZ)$ whose image modulo $N$ is of the form $\left(\begin{smallmatrix} h^{-1} & * \\ 0 & h\end{smallmatrix}\right)$ for some $h\in H$.   By (\ref{E:newform decomp 2}), we have
\[
S_k(\Gamma)= \bigoplus_{M|N} \bigoplus_{\substack{f\in \calN'(k,M)\\ \varepsilon_f(H)=1}} \bigoplus_{\substack{d|\frac{N}{M} \\ d\leq (N/M)^{1/2}}}  \calM_{f,d}.
\]
The following proposition is now an immediate consequence of Lemmas~\ref{L:Mfd conditions} and \ref{L:M reduction}.  

\begin{prop} \label{P:cusp condition check}
The space $S_k(\Gamma)$ satisfies conditions (\ref{c:a})--(\ref{c:d}) of \S\ref{SS:setup intro}.
\end{prop}

\subsection{Eisenstien series} \label{SS:Eisenstein}
Fix a congruence subgroup $\Gamma_1(N)\subseteq \Gamma \subseteq \Gamma_0(N)$ and let $H$ be a subgroup of $(\ZZ/N\ZZ)^\times$ associated to $\Gamma$ as in \S\ref{SS:cusp forms}.  Let $E_k(\Gamma)$ be the Eisenstein subspace of $M_k(\Gamma)$.    We have
\[
M_k(\Gamma)=E_k(\Gamma)\oplus S_k(\Gamma).
\]

  Let $\chi_1$ and $\chi_2$ be primitive Dirichlet characters modulo $N_1$ and $N_2$, respectively, satisfying $(\chi_1\chi_2)(-1)=(-1)^k$.    For each integer $e\geq 1$,  define the \defi{Eisenstein series}
\[
F_{k}(\chi_1,\chi_2, e)(\tau) = c_0 + \sum_{n=1}^\infty \sigma_{k-1}(\chi_1,\chi_2,n) q^{ne},
\]
where 
\[
\sigma_{k-1}(\chi_1,\chi_2,n) := \sum_{d|n, d\geq 1} d^{k-1} \chi_1(d)\chi_2(n/d),
\]
$c_0=0$ if $N_2\neq 1$, and $c_0=-B_{k,\chi_1}/(2k)$ if $N_2=1$ (where $B_{k,\chi}$ is a generalized Bernoulli number, cf.~\S5 of \cite{MR2289048}).
  
Except for the case with $k=2$ and $\chi_1=\chi_2=1$, $F_{k}(\chi_1,\chi_2, e)$ is an element of $M_k(\Gamma_1(N_1N_2e))$.   Note that $F_{k}(\chi_1,\chi_2, e)$ satisfies $F_{k}(\chi_1,\chi_2, e)|\ang{d} = (\chi_1\chi_2)(d) \cdot F_{k}(\chi_1,\chi_2, e)$ for all $d\in (\ZZ/N_1N_2 e\ZZ)^\times$.  If $k=2$, $\chi_1=\chi_2=1$ and $e>1$, then $F_{k}(\chi_1,\chi_2, 1) - e F_{k}(\chi_1,\chi_2, e)$ is an element of $M_2(\Gamma_0(e))$.     

The set $B$ of Eisenstein series as above with $N_1 N_2 e$ dividing $N$ and $(\chi_1\chi_2)(H)=1$ form a basis of $E_k(\Gamma)$; this follows from Theorem~5.9 of \cite{MR2289048}.

  We can compute an arbitrary number of terms of the $q$-expansion of modular forms in $B$.   We also know the action of the diamond operators with respect to the basis $B$.   For each $f\in B$, we can also compute arbitrarily many terms in the $q$-expansion of $f|W_N$; for example, see \S2.2 of \cite{MR3889557}.

\begin{prop}
The space $M_k(\Gamma)$ satisfies conditions (\ref{c:a})--(\ref{c:d}) of \S\ref{SS:setup intro}.
\end{prop}
\begin{proof}
The space $M_k(\Gamma)$ is indeed stable under the actions of $W_N$ and the diamond operators.   By Proposition~\ref{P:cusp condition check}, there is a basis $B'$ of $S_k(\Gamma,\ZZ)$ that spans $S_k(\Gamma)$ and for which we can compute arbitrarily many terms of their $q$-expansions.   The set $B\cup B'$ spans $M_k(\Gamma)$.   By Lemmas~\ref{L:gen calM} and \ref{L:finite index SNF}, we can find a basis of $M_k(\Gamma,\ZZ)$ that spans $M_k(\Gamma)$ and for which we can compute arbitrarily many terms of their $q$-expansions. 
\end{proof}

\subsection{Numerically approximations for the Atkin--Lehner action} \label{SS:numerical}

Take any modular form $h\in M_k(\Gamma_1(N),\ZZ)$ for which we can compute arbitrarily many terms of its $q$-expansion.  We now explain how we can approximate each coefficient of $h|W_N$ to arbitrary accuracy in $\CC$.  In the setting of \S\ref{SS:setup intro}, this will allow us to approximate the entries of the matrix $W$ to arbitrary accuracy.

Let $B$ be the basis of $E_k(\Gamma_1(N))$ from \S\ref{SS:Eisenstein}.  Define 
\[
B':= \bigcup_{M|N,\, d|(N/M)} \{ \alpha_d(f) : f\in \calN(k,M) \};
\]
The set $B'$ is a basis of $S_k(\Gamma_1(N))$ from (\ref{SS:Eisenstein}) and hence $B\cup B'$ is a basis of $M_k(\Gamma_1(N))$.  For each $f\in B\cup B'$, every coefficient $a_n(f)$ in $\CC$ is algebraic and can be computed to arbitrary accuracy.    Expressing $h$ in terms of the basis $B\cup B'$, it suffices to explain how for each $f\in B\cup B'$ we can approximate any coefficient of $f|W_N$ to arbitrary accuracy in $\CC$.

If $f\in B$, we can compute arbitrarily many terms of the $q$-expansion of $f|W_N$, cf.~ \S2.2 of \cite{MR3889557}.  Finally consider any element in $B'$; it is of the form $\alpha_d(f)$ with $f\in \calN(k,M)$.   We can compute any coefficient of $\alpha_d(f)$ to arbitrary accuracy in $\CC$ by using Lemma~\ref{L:W action on eigenbasis} and Proposition~\ref{P:W action on newforms} (approximations of $\lambda_M(f)$ can be found as in \S\ref{SS:pseudo approx}).

 \section{Verification of the algorithm} \label{S:verification}
 We now verify the validity of the algorithm from \S\ref{SS:the algorithm}.   Recall that $Q$ is the smallest positive divisor of $N$ for which the action of $\ang{d}$ on $\calM$ depends only on the value of $d$ modulo $Q$.    By Theorem~\ref{T:AL arithmetic}(\ref{T:AL arithmetic ii}), we find that $\ang{d}$ acts on $\calM(\ZZ)$ with inverse $\ang{d^{-1}}$.    Therefore, each $D_d$ lies in $\GL_g(\ZZ)$.
 
\begin{lemma} \label{L:verification B}
For any $f\in \calM(\ZZ)$, the modular form $B_{k,N}\cdot f|W_N$ has Fourier coefficients in $\ZZ[\zeta_Q]$.
\end{lemma}
\begin{proof}
Fix a modular form $f\in \calM(\ZZ)$.   Take any automorphism $\sigma$ of $\CC$ that fixes $\QQ(\zeta_Q)$.  By Theorem~\ref{T:AL arithmetic}(\ref{T:AL arithmetic iii}) and using that $f$ has rational  coefficients, we find that 
\[
\sigma(f|W_N) =  (f|W_N)|\ang{\chi_N(\sigma)} =f|W_N,
\]
where the last equality uses that $f|W_N\in \calM$ and that $\chi_N(\sigma)\equiv 1 \pmod{Q}$ since $\sigma$ fixes $\zeta_Q$.   We deduce that $f|W_N$, and hence also $B_{k,N}\cdot f|W_N$, has coefficients in $\QQ(\zeta_Q)$ since $\sigma$ is an arbitrary automorphism of $\CC$ that fixes $\QQ(\zeta_Q)$.   By Theorem~\ref{T:integrality}(\ref{T:integrality i}), the Fourier coefficients of $B_{k,N}\cdot f|W_N$ are algebraic integers.  We deduce that $B_{k,N}\cdot f|W_N$ has coefficients in $\ZZ[\zeta_Q]$ which is the ring of integers of $\QQ(\zeta_Q)$.  
\end{proof}

\begin{lemma} \label{L:verification alpha}
For any $f\in \calM(\ZZ[\zeta_Q])$, we have $\alpha f= b_1 f_1 + \cdots b_g f_g$ with $b_i \in \ZZ[\zeta_Q]$.
\end{lemma}
\begin{proof}
There is no harm in changing the basis $f_1,\ldots, f_g$ of $\calM(\ZZ)$.   In particular, we may assume $f_1,\ldots, f_g$ are chosen so that the matrix $A$ of \S\ref{SS:the algorithm} is in Hermite normal form.  Let $a_j$ be the leading coefficient of the $q$-expansion of $f_j$.   The pivots of the matrix $A$ are $a_1,\ldots, a_g$ and hence $\alpha=a_1\cdots a_g$.    We have $ f= c_1 f_1 + \cdots + c_g f_g$ for unique $c_i \in \QQ(\zeta_Q)$.

We claim that $a_1\cdots a_i \cdot c_i$ is an element of $\ZZ[\zeta_Q]$ for all $1\leq i \leq g$.   Suppose that the claim is false and let $1\leq j \leq g$ be the minimal value for which $a_1\cdots a_j \cdot c_j \notin \ZZ[\zeta_Q]$.     Therefore,
\begin{align} \label{E:leading coefficient alpha}
\sum_{i=j}^{g} (a_1\cdots a_{j-1} \, c_i) \, f_i = a_1\cdots a_{j-1} \cdot f - \sum_{i=1}^{j-1} (a_1\cdots a_{j-1} \, c_i) \, f_i
\end{align}
has coefficients in $\ZZ[\zeta_Q]$, where we have used that $f$ and the $f_i$ have coefficients in $\ZZ[\zeta_Q]$ and the minimality of $j$.  The leading coefficients of (\ref{E:leading coefficient alpha}) is $a_1\cdots a_{j-1} c_j \cdot a_j$ since we have chosen the basis $f_1,\ldots, f_g$ so that $A$ is in Hermite normal form.  So $a_1\cdots a_j \cdot c_j \in \ZZ[\zeta_Q]$ which contradicts the choice of $j$ and thus proves the claim.

By the above claim, we have $\alpha c_i \in \ZZ[\zeta_Q]$ for all $1\leq i \leq g$ since $\alpha=a_1\cdots a_g$.  This proves the lemma with $b_i=\alpha c_i$.
\end{proof}

\begin{lemma} \label{L:integrality W matrix}
The matrix $B_{k,N} \alpha\cdot W$ lies in $M_g(\ZZ[\zeta_Q])$.
\end{lemma}
\begin{proof}
Take any $1\leq j \leq g$.  By Lemma~\ref{L:verification B} and our assumption that $\calM$ is stable under the action of $W_N$, we have $B_{k,N}\cdot f_j|W_N \in \calM(\ZZ[\zeta_Q])$.   By Lemma~\ref{L:verification alpha}, we deduce that $B_{k,N} \alpha\cdot f_j|W_N = \sum_{i=1}^g b_{j,i} f_i$ with $b_{j,i} \in \ZZ[\zeta_Q]$.  From the definition of the matrix $W$, we find that $B_{k,N}\alpha\cdot W_{j,i} = b_{j,i}$ for all $1\leq i \leq g$.  Since $j$ was arbitrary, we deduce that the entries of $B_{k,N}\alpha\cdot W$ lies in $\ZZ[\zeta_Q]$.
\end{proof}

We now describe the natural Galois action on the matrix $W$.

\begin{lemma} \label{L:sigma W D}
For each $\sigma\in \Gal(\QQ(\zeta_Q)/\QQ)$, we have $\sigma(W)=  W\cdot D_{\chi_Q(\sigma)}$.
\end{lemma}
\begin{proof}
Take any $1\leq j \leq g$.   By the definition of the matrix $W$, we have $f_j|W_N = \sum_{k=1}^g  W_{j,k}\cdot f_k$.  Therefore, 
\[
\sigma(f_j|W_N) = \sum_{k=1}^g  \sigma(W_{j,k})\cdot \sigma(f_k)= \sum_{k=1}^g  \sigma(W_{j,k})\cdot f_k,
\]
where we have used that each $f_k$ has coefficients in $\ZZ$.   Set $d:=\chi_N(\sigma) \in (\ZZ/N\ZZ)^\times$.  By Theorem~\ref{T:AL arithmetic}(\ref{T:AL arithmetic iii}) and using that $f_j$ has integer coefficients, we have $\sigma(f_j|W_N) = (f_j|W_N)|\ang{d}$.   Using the definition of $W$ and $D_d$, we deduce that
\begin{align*}
\sigma(f_j|W_N) = (f_j|W_N)|\ang{d}  = \sum_{i=1}^g  W_{j,i}\cdot f_i|\ang{d} = \sum_{i=1}^g  W_{j,i}\cdot \sum_{k=1}^g (D_d)_{i,k}\cdot f_k = \sum_{k=1}^g  (W D_d)_{j,k}\cdot f_k.
\end{align*}
By comparing our two expressions for $\sigma(f_j|W_N)$, we find that $\sigma(W)_{j,k}=(W D_d)_{j,k}$ for all $1\leq k \leq g$.   Since $1\leq j \leq g$ was arbitrary, we conclude that $\sigma(W)= W D_d$.  Finally, we note that $D_d$ depends only on $d$ modulo $Q$ and $d\equiv \chi_Q(\sigma) \pmod{Q}$.
\end{proof}

We have $W\in M_g(\QQ(\zeta_Q))$ by Lemma~\ref{L:integrality W matrix}.  For each integer $0 \leq b \leq \varphi(Q)-1$, define the matrix  
\begin{align} \label{E:betab new}
\beta_b:= \Tr_{\QQ(\zeta_Q)/\QQ}( \zeta_Q^b\,  W ) \in M_g(\QQ).
\end{align}
By Lemma~\ref{L:integrality W matrix}, we have $B_{k,N}\alpha \, W \in M_g(\ZZ[\zeta_Q])$ and hence $B_{k,N} \alpha\cdot \beta_b \in M_g(\ZZ)$.  We have
\begin{align*}
\beta_b = \sum_{\sigma \in \Gal(\QQ(\zeta_Q)/\QQ)} \sigma(\zeta_Q)^b \cdot \sigma(W) =\sum_{\sigma \in \Gal(\QQ(\zeta_Q)/\QQ)}  \zeta_Q^{\chi_Q(\sigma)\cdot b}\cdot W \cdot D_{\chi_Q(\sigma)},
\end{align*}
where the last equality uses Lemma~\ref{L:sigma W D}.  Therefore,
\[
\beta_b =  W \cdot \sum_{d\in (\ZZ/Q\ZZ)^\times} \zeta_Q^{db} \, D_d
\]
which agrees with the definition of $\beta_b$ given in \S\ref{SS:the algorithm}.  

In \S\ref{SS:numerical}, we explained how to numerical approximate the matrices $W$ and $D_d$ in $M_g(\CC)$.  Since $D_d$ has integer entries, it can be determined by a sufficiently accurate approximation.   So an approximation of $W$ gives an approximation of the matrix $\beta_b$.   Since $N_{k,N}\alpha\, \beta_b$ has integer entries, we can thus determine $\beta_b$ from a sufficiently accurate approximation of $W$ in $M_g(\CC)$.

Finally, in \S\ref{SS:the algorithm}, we observed that $W$ is the unique matrix in $M_g(\QQ(\zeta_Q))$ that satisfies $\Tr_{\QQ(\zeta_Q)/\QQ}(\zeta_Q^b W)= \beta_b$ for all $0\leq b \leq \varphi(Q)-1$.  Moreover, it is straightforward to compute $W$ given the matrices $\beta_b$.

\section{Modular curves} \label{S:modular curves}

Fix a positive integer $N\geq 1$.   The group $\SL_2(\ZZ)$ acts on the upper half plane $\mathfrak{H}$ via linear fractional transformations.   The quotient $\Gamma(N)\backslash \mathfrak{H}$ is a Riemann surface and can be completed to a compact and smooth Riemann surface $\calX_N$. 

Every meromorphic function $f$ on $\calX_N$ has a $q$-expansion  $\sum_{n\in \ZZ} c_n(f) q_N^{n}$, where the $c_n(f)\in \CC$ are $0$ for all but finitely many negative $n \in \ZZ$.  Let $\calF_N$ be the field consisting of all meromorphic functions $f$ on $\calX_N$ for which $c_n(f)$ lies in $\QQ(\zeta_N)$ for all $n$.  For example, $\calF_1=\QQ(j)$, where $j$ is the modular $j$-invariant.

\begin{lemma} \label{L:Shimura basic}
There is a unique right action $*$ of $\GL_2(\ZZ/N\ZZ)$ on the field $\calF_N$ such that the following hold for all $f\in \calF_N$:
\begin{alphenum}
\item
For $A\in \SL_2(\ZZ/N\ZZ)$, we have $(f*A)(\tau) = f(\gamma\tau)$, where $\gamma\in \SL_2(\ZZ)$ is any matrix congruent to $A$ modulo $N$.
\item
For $A=\left(\begin{smallmatrix}1 & 0 \\0 & d\end{smallmatrix}\right) \in \GL_2(\ZZ/N\ZZ)$, the $q$-expansion of $f*A$ is $\sum_{n\in \ZZ} \sigma_d(c_n(f)) q_N^{n}$, where $\sigma_d$ is the automorphism of the field $\QQ(\zeta_N)$ that satisfies $\sigma_d(\zeta_N)=\zeta_N^d$.
\end{alphenum}
\end{lemma}
\begin{proof}
This follows from Theorem~6.6 and Proposition~6.9 of \cite{MR1291394}.
\end{proof}

For a subgroup $G$ of $\GL_2(\ZZ/N\ZZ)$, let $\calF_N^G$ be the subfield of $\calF_N$ fixed by $G$ under the action of Lemma~\ref{L:Shimura basic}.

\begin{lemma} \label{L:FN basics}
\begin{romanenum}
\item \label{L:FN basics i}
The matrix $-I$ acts trivially on $\calF_N$ and the right action of $\GL_2(\ZZ/N\ZZ)/\{\pm I\}$ on $\calF_N$ is faithful.
\item \label{L:FN basics ii}
We have $\calF_N^{\GL_2(\ZZ/N\ZZ)}=\calF_1=\QQ(j)$ and $\calF_N^{\SL_2(\ZZ/N\ZZ)}=\QQ(\zeta_N)(j)$.
\item \label{L:FN basics iii}
The field $\QQ(\zeta_N)$ is algebraically closed in $\calF_N$.
\end{romanenum}
\end{lemma}
\begin{proof}
This also follows from Theorem~6.6 and Proposition~6.9 of \cite{MR1291394}.
\end{proof}

Let $G$ be a subgroup of $\GL_2(\ZZ/N\ZZ)$ that satisfies $\det(G)=(\ZZ/N\ZZ)^\times$ and $-I\in G$.  By Lemma~\ref{L:FN basics}, the field $\calF_N^G$ has transcendence degree $1$ and $\QQ$ is algebraically closed in $\calF_N^G$.   

We define the \defi{modular curve} $X_G$ to be the smooth, projective and geometrically irreducible curve over $\QQ$ with function field $\calF_N^G$.   We can identify $X_G(\CC)$ with the compact and smooth Riemann surface that completes $\Gamma_G\backslash \mathfrak{H}$, where $\Gamma_G$ is the congruence subgroup consisting of matrices in $\SL_2(\ZZ)$ whose image modulo $N$ lies in $G$.\\

We now describe how the modular curves $X_G$ are related to understanding the Galois action on the torsion points of elliptic curves; this is given for background purposes and will not be used elsewhere in the paper.  Let $\pi_G \colon X_G \to \Spec \QQ[j] \cup \{\infty\}=\PP^1_\QQ$ be the morphism arising from the inclusion $\QQ(j) \subseteq \calF_N^G$.  In particular, we may view $\pi_G(X_G(\QQ))$ as a subset of $\QQ \cup\{\infty\}$.

Consider an elliptic curve $E/\QQ$ whose $j$-invariant we denote by $j_E\in \QQ$.  Let $E[N]$ be the $N$-torsion subgroup of $E(\Qbar)$, where $\Qbar$ is a fixed algebraic closure of $\QQ$.   The group $E[N]$ is a free $\ZZ/N\ZZ$-module of rank $2$ and has a natural action of $\Gal_\QQ:=\Gal(\Qbar/\QQ)$ that respects the group structure.  By choosing a basis for $E[N]$, the Galois action can be expressed by a representation
\[
\rho_{E,N}\colon \Gal_\QQ \to \GL_2(\ZZ/N\ZZ).
\]
The subgroup $\rho_{E,N}(\Gal_\QQ)$ of $\GL_2(\ZZ/N\ZZ)$ is uniquely defined up to conjugacy.

Let $G^t$ be the subgroup of $\GL_2(\ZZ/N\ZZ)$ obtained by taking the transpose of the elements of $G$.  Suppose further that $j_E\notin \{0,1728\}$; equivalently, the automorphism group of the elliptic curve $E_{\Qbar}$ is cyclic of order $2$.  Then $\rho_{E,N}(\Gal_\QQ)$ is conjugate in $\GL_2(\ZZ/N\ZZ)$ to a subgroup of $G^t$ if and only if $j_E$ is an element of $\pi_G(X_G(\QQ))$, cf.~Proposition~3.3 of \cite{Zywina-possible} (the transpose arises because the action in Proposition~3.1 of \cite{Zywina-possible} is slightly different than ours).  So the modular curves $X_G$, and their morphisms $\pi_G$, contain information about the images of $\rho_{E,N}$ for elliptic curves $E/\QQ$ with $j_E \notin \{0,1728\}$.

\begin{remark}
\begin{romanenum}
\item The assumptions $\det(G)=(\ZZ/N\ZZ)^\times$ and $-I\in G$ are natural in this elliptic curve setting.    We have $\det(\rho_{E,N}(\Gal_\QQ))=(\ZZ/N\ZZ)^\times$ and the group $\ang{\rho_{E,N}(\Gal_\QQ),-I}$, up to conjugacy in $\GL_2(\ZZ/N\ZZ)$, depends only on $j_E$.
\item
The occurrence of $G^t$ is due to the fact that in this paper we have natural right actions while we usually view the action of $\Gal_\QQ$ on $E[N]$ as a left action.    Sometimes in the literature, for example in \cite{Zywina-possible}, the modular curve corresponding to the group $G$ agrees with our $X_{G^t}$.   
\end{romanenum}
\end{remark}

\subsection{The canonical ring}
This section is dedicated to describing the canonical ring $R_{X_G}:= \bigoplus_{k\geq 0} H^0(X_G, \Omega_{X_G}^{\otimes k})$ of $X_G$, and in particular $H^0(X_G, \Omega_{X_G})$, in terms of cusps forms.  This is certainly well-known, but lacking a reference we give a quick demonstration.

We now fix an integer $k\geq 0$.   Take any $\omega \in H^0(X_G,\Omega^{\otimes k}_{X_G})$.   The form $\omega$ induces a differential $k$-form on $X_G(\CC)$; on $\mathfrak{H}$ it equals $(2\pi i)^{k} f(\tau) \, (d\tau)^k$ for a unique cusp form $f\in S_{2k}(\Gamma_G,\CC)$.  We define $\alpha_k(\omega):=f$.

In \S\ref{SS:intro modular curves}, we defined a right action of $\GL_2(\ZZ/N\ZZ)$ on $S_{2k}(\Gamma(N),\QQ(\zeta_N))$; we also denote it by $*$.

\begin{lemma} \label{L:alphak image}
The cusp form $\alpha_k(\omega)$ lies in $S_{2k}(\Gamma(N), \QQ(\zeta_N))^G$ for all $\omega\in H^0(X_G,\Omega^{\otimes k}_{X_G})$.
\end{lemma}
\begin{proof}
Take any $\omega \in H^0(X_G,\Omega^{\otimes k}_{X_G})$ and set $f:=\alpha_k(\omega) \in S_{2k}(\Gamma_G,\CC)$.   Choose a non-constant $u\in \calF_N^G$.   We have $\omega = v (du)^k$ for a unique modular function $v\in \calF_N^G$.   The form $\omega$ on $X_G(\CC)$ arises from the form $v(\tau) u'(\tau)^k (d\tau)^k$ on $\mathfrak{H}$.  Therefore, $f(\tau)=(2\pi i)^{-k} v(\tau) u'(\tau)^k$.

We claim that the coefficients of the $q$-expansion of $f$ lie in $\QQ(\zeta_N)$.  The coefficients of the $q$-expansion of $v$ are in $\QQ(\zeta_N)$ since $v\in \calF_N$.  So to prove the claim, it suffices to show that $(2\pi i)^{-1} u'(\tau)$ has a $q$-expansion with coefficients in $\QQ(\zeta_N)$.   Since $u\in \calF_N$, we have $u=\sum_{n \in \ZZ} c_n(u) q_N^{n}$ with $c_n(u)\in \QQ(\zeta_N)$ that are $0$ for all but finitely many negative $n$.  Therefore, $(2\pi i)^{-1} u'(\tau) = \sum_{n \in \ZZ} n/N \cdot c_n(u) q_N^{n}$ which has coefficients in $\QQ(\zeta_N)$.  This proves the claim.

Now take any $A\in G$. Set $d=\det(A)$ and choose $\gamma\in \SL_2(\ZZ)$ so that $A\equiv \gamma  \left(\begin{smallmatrix}1 & 0 \\0 & d\end{smallmatrix}\right) \pmod{N}$.  Since $u\in \calF_N^G$, we have $u(\tau)=(u*A)(\tau) = \sigma_d( u(\gamma\tau))$.   The coefficients of $u(\gamma\tau)$ lie in $\QQ(\zeta_N)$ since $u\in \calF_N$ and $\calF_N$ is stable under the right action of $\GL_2(\ZZ/N\ZZ)$.  We have $u(\gamma \tau) = \sum_n b_n q_N^{n}$ with $b_n \in \QQ(\zeta_N)$.  Taking derivatives gives
\begin{align*} 
(2\pi i)^{-1}  u'(\tau) &= (2\pi i)^{-1} \frac{d}{d\tau} \sigma_d( u(\gamma\tau))  = (2\pi i)^{-1} \frac{d}{d\tau}  \sum_n \sigma_d(b_n) q_N^{n}= \sum_n n/N\cdot \sigma_d(b_n) q_N^{n}. 
\end{align*}
Therefore,
\begin{align} \label{E:d  and sigmad commute}
(2\pi i)^{-1}  u'(\tau) = \sigma_d\Big(\sum_n n/N\cdot b_n q_N^{n}\Big)  = \sigma_d\Big((2\pi i)^{-1} \frac{d}{d\tau}u(\gamma\tau)\Big) = \sigma_d\Big( (2\pi i)^{-1} (u' |_2 \gamma)(\tau) \Big).
\end{align}
Since $v\in \calF_N^G$, we have $v(\tau) = (v*A)(\tau)=\sigma_d( v(\gamma \tau))$.  Taking the $k$-power of both sides of (\ref{E:d  and sigmad commute}) and multiplying by $v(\tau)$ gives
\[
(2\pi i)^{-k} v(\tau) u'(\tau)^k = \sigma_d\big( (2\pi i)^{-k} v(\tau) (u' |_2 \gamma)^k (\tau) \big) = \sigma_d\Big( \big( ((2\pi i)^{-k} v\,  (u')^k) |_{2k} \gamma\big) (\tau) \Big);
\]
equivalently, $f = \sigma_d( f|_{2k} \gamma)$.  Therefore, $f=f*A$.   Since $A$ was an arbitrary element of $G$, we deduce that $f\in S_{2k}(\Gamma(N),\QQ(\zeta_N))^G$.
\end{proof}

Using Lemma~\ref{L:alphak image}, we have a linear map
\[
\alpha_k\colon H^0(X_G,\Omega^{\otimes k}_{X_G}) \to S_{2k}(\Gamma(N),\QQ(\zeta_N))^G
\]
of $\QQ$-vector spaces.  

\begin{lemma} \label{L:injectivity of alphak}
The linear map $\alpha_k$ is injective for all $k$.  The linear map $\alpha_1$ is an isomorphism.
\end{lemma}
\begin{proof}
Define $H=G\cap \SL_2(\ZZ/N\ZZ)$.   Since $H$ is the image of  $\Gamma_G$ modulo $N$, we have $S_{2k}(\Gamma(N),\QQ(\zeta_N))^H = S_{2k}(\Gamma_G,\QQ(\zeta_N))$.    In particular, $S_{2k}(\Gamma(N),\QQ(\zeta_N))^G \subseteq S_{2k}(\Gamma_G,\CC)$.   Let
\[
\iota\colon S_{2k}(\Gamma(N),\QQ(\zeta_N))^G\otimes_\QQ \CC  \to S_{2k}(\Gamma_G,\CC) 
\]
be the $\CC$-linear map induced by the inclusion.

We claim that $\iota$ is an isomorphism.  The group $H$ is normal in $G$, so we obtain a right action of $G/H$ on $S_{2k}(\Gamma_G,\QQ(\zeta_N))$.  Let $\varphi\colon G\to \Gal(\QQ(\zeta_N)/\QQ)$ be the homomorphism satisfying $\varphi(A)(\zeta_N)=\zeta_N^{\det A}$; it is surjective since $\det(G)=(\ZZ/N\ZZ)^\times$.  We obtain an action of $\Gal(\QQ(\zeta_N)/\QQ)$ on $S_{2k}(\Gamma_G,\QQ(\zeta_N))$ by using the action of $G/H$ and the isomorphism $G/H\xrightarrow{\sim} \Gal(\QQ(\zeta_N)/\QQ)$ induced by $\varphi$; note that we can view this as a left action $\bullet$ since $G/H$ is abelian.   With this new action,  we have $\sigma \bullet (c f )= \sigma(c) \, (\sigma \bullet f)$ for all $c\in \QQ(\zeta_N)$, $f\in S_{2k}(\Gamma_G,\QQ(\zeta_N))$ and $\sigma\in \Gal(\QQ(\zeta_N)/\QQ)$.  By Galois descent for vector spaces (see the corollary to Proposition~6 in Chapter V \S10 of \cite{MR1994218}), the natural homomorphism 
\begin{align}\label{E:S2k descent}
S_{2k}(\Gamma(N),\QQ(\zeta_N))^G\otimes_\QQ \QQ(\zeta_N) = S_{2k}(\Gamma_G,\QQ(\zeta_N))^{\Gal(\QQ(\zeta_N)/\QQ)} \otimes_\QQ \QQ(\zeta_N) \to S_{2k}(\Gamma_G,\QQ(\zeta_N)) 
\end{align}
is an isomorphism of $\QQ(\zeta_N)$-vector spaces.  Since $S_{2k}(\Gamma_G,\CC)$ has a basis with coefficients in $\QQ(\zeta_N)$, we deduce that $\iota$ is an isomorphism by tensoring (\ref{E:S2k descent}) up to $\CC$.

By tensoring $\alpha_k$ up to $\CC$ and composing with $\iota$, we obtain a $\CC$-linear map
\[
\beta_k\colon H^0(X_G(\CC),\Omega^{\otimes k}_{X_G(\CC)})=H^0(X_G,\Omega^{\otimes k}_{X_G})\otimes_\QQ \CC \to S_{2k}(\Gamma_G,\CC).
\]
Since $\iota$ is an isomorphism, it suffices to prove that $\beta_k$ is injective and that $\beta_1$ is an isomorphism.   For any holomorphic $k$-form $\omega$ on $X_G(\CC)$, $f:=\beta_k(\omega)$ is the \emph{unique} cusp form in $S_{2k}(\Gamma_G,\CC)$ such that the form on $\mathfrak{H}$ induced by $\omega$ equals $(2\pi i)^k f(\tau) (d\tau)^k$.    So $\beta_k$ is indeed injective.   The linear map $\beta_1$ is an isomorphism, cf.~Corollary 2.17 of \cite{MR1291394}.
\end{proof}

Let $R_{X_G}= \bigoplus_{k\geq 0} H^0(X_G, \Omega_{X_G}^{\otimes k})$ be the canonical ring of $X_G$.   Using the linear maps $\alpha_k$, we obtain a homomorphism 
\begin{align} \label{E:graded isomorphism alpha}
\alpha\colon R_X \to \bigoplus_{k\geq 0} S_{2k}(\Gamma_G,\QQ(\zeta_N))^G
\end{align}
of graded rings, where multiplication on the right hand side is multiplication of functions.   From Lemma~\ref{L:injectivity of alphak}, the homomorphism $\alpha$ is injective, and $\alpha_1\colon H^0(X_G, \Omega_{X_G}) \to S_{2}(\Gamma_G,\QQ(\zeta_N))^G$ is an isomorphism.

\section{Canonical map} \label{S:canonical}

\subsection{Setup}  \label{SS:setup}
Fix an integer $N\geq 1$ and a subgroup $G$ of $\GL_2(\ZZ/N\ZZ)$ that satisfies $-I\in G$ and $\det(G)=(\ZZ/N\ZZ)^\times$.    

From \S\ref{SS:SL2 action}, and the algorithm of \S\ref{SS:the algorithm}, we can compute a basis of the $\QQ$-vector space $S_2(\Gamma(N),\QQ(\zeta_N))$ and, with respect to this basis, compute the right action of $\SL_2(\ZZ)$.  Using the action of $\GL_2(\ZZ/N\ZZ)$ on $S_2(\Gamma(N),\QQ(\zeta_N))$ from \S\ref{SS:intro modular curves}, one can then compute a basis $f_1,\ldots, f_g$ of the $\QQ$-vector space $S_2(\Gamma(N),\QQ(\zeta_N))^G$.  Each $f_j$ is given by a $q$-expansion for which an arbitrary number of its coefficients can be computed.   

Let $\omega_1,\ldots, \omega_g$ be the basis of the $\QQ$-vector space $H^0(X_G,\Omega_{X_G}^1)$ that satisfies $\alpha_1(\omega_j)=f_j$, where $\alpha_1$ is the isomorphism $H^0(X_G,\Omega_{X_G}) \xrightarrow{\sim} S_{2}(\Gamma(N),\QQ(\zeta_N))^G$ from \S\ref{S:modular curves}.   Observe that $g$ is the genus of the modular curve $X_G$.     We shall assume that $g\geq 2$.  

Let 
\[
\varphi\colon X_G \to \PP^{g-1}_\QQ
\]
be the canonical morphism corresponding to the basis $\omega_1,\ldots, \omega_g$ and denote its image by $C$.   The goal of \S\ref{S:canonical} is to describe how to compute the ideal $I(C) \subseteq \QQ[x_1,\ldots, x_g]$ of $C$, and hence the curve $C\subseteq \PP^{g-1}_\QQ$, from the cusps forms $f_1,\ldots, f_g$.   If $X_G$ is not hyperelliptic, then $\varphi$ will be an embedding and hence we will have found a model for $X_G$.

\subsection{Background}

Let $X$ be a smooth, projective and geometrically irreducible curve defined over a field $k$.  In our application, the curve $X$ will be the modular curve $X_G$ defined over $\QQ$.  Denote the genus of $X$ by $g$ and assume that $g\geq 2$.   Fix a basis $\omega_1,\ldots, \omega_g$ of the $k$-vector space $H^0(X,\Omega^1_X)$; it gives rise to a non-constant morphism
\[
\varphi\colon X\to \PP^{g-1}_k.
\] 

Define the curve $C:=\varphi(X)$ and let $I(C) \subseteq k[x_1,\ldots, x_g]$ be the homogeneous ideal of $C$.   We have $I(C)=\oplus_{d\geq 0} I_d(C)$, where $I_d(C)$ consists of the homogeneous polynomials in $I(C)$ of degree $d$.   

 We say that $X$ is \emph{hyperelliptic} if there is a morphism $X_{\kbar} \to \PP^1_{\kbar}$ of degree $2$ (by the following lemmas, this is equivalent to there being a morphism $X\to Y$ of degree $2$ with $Y$ a curve of genus $0$).  We first consider the case where $X$ is not hyperelliptic. 

\begin{lemma} \label{L:canonical map, not hyperelliptic}
Suppose that $X$ is not hyperelliptic.  
\begin{romanenum}
\item \label{L:canonical map, not hyperelliptic i}
The morphism $\varphi$ is an embedding.  In particular, $X$ and $C$ are isomorphic.
\item \label{L:canonical map, not hyperelliptic ii}
We have $\dim_k I_2(C) = (g-2)(g-3)/2$ and $\dim_k I_3(C) = (g - 3)(g^2 + 6g - 10)/6$.
\item \label{L:canonical map, not hyperelliptic iii}
If $g\geq 4$, then the ideal $I$ is generated by $I_2(C) $ and $I_3(C)$.
\item \label{L:canonical map, not hyperelliptic iv}
If $g =3$, then the ideal $I$ is generated by $I_4(C)$ and $\dim_k I_4(C)=1$.
\end{romanenum}
\end{lemma}
\begin{proof}
First assume that $k$ is algebraically closed.   The morphism $\varphi$ is an embedding and the curve $C\subseteq \PP_k^{g-1}$ has degree $2g-2$, cf.~\cite{MR0463157}*{IV~\S5}.  Let $H$ be a hyperplane section of $C\subseteq \PP^{g-1}_k$.  Fix an integer $d\geq 2$.  We have $\deg(dH)=d(2g-2) > 2g-2$ and hence $l(dH)=d(2g-2)-g+1$ by Riemann--Roch.  The $d$-th component of the graded ring $k[x_1,\ldots, x_g]/I(C)$ has dimension $l(dH)$ and hence $\dim_k I_d(C) = \binom{g-1+d}{d}-l(dH) = \binom{g-1+d}{d}- d(2g-2)+g-1$.  The claimed dimensions in the lemma are now immediate.  Part (\ref{L:canonical map, not hyperelliptic iii}) is a theorem of Petri, cf.~\cite{MR943539}.    Finally suppose that $g=3$ and let $F$ be a generator of the vector space $I_4(C)$.  Since $C$ is a smooth curve of genus $3$, this implies that $C$ is defined by $F=0$ and hence $I(C)$ is generated by $F$.

Now consider a general field $k$.  The forms $\omega_1,\ldots, \omega_g$ are also a basis of $H^0(X_{\kbar},\Omega_{X_{\kbar}}^1)=H^0(X,\Omega_{X}^1)\otimes_k \kbar$.  With this basis fixed, the natural map $I_d(C) \otimes_k \kbar \to I_d(C_{\kbar})$ is an isomorphism for all $d\geq 0$.  It is now easy to deduce the lemma from the algebraically closed case.
\end{proof}

\begin{lemma}  \label{L:canonical map, hyperelliptic}
Suppose that $X$ is hyperelliptic.   
\begin{romanenum}
\item
The curve $C$ has genus $0$ and $X\xrightarrow{\varphi} C$ has degree $2$.
\item
The ideal $I(C)$ is generated by $I_2(C)$ and $\dim_k I_2(C) = (g-1)(g-2)/2$.   
\end{romanenum}
\end{lemma}
\begin{proof}
As in the proof of Lemma~\ref{L:canonical map, not hyperelliptic}, the natural map $I_d(C)\otimes_k \kbar \to I_d(C_{\kbar})$ is an isomorphism for all $d\geq 0$.  It is now easy to show that the general case reduced to the case where $k$ is algebraically closed.

Assume $k$ is algebraically closed.   The morphism $\varphi\colon X\to C$ has degree $2$ and the curve $C\subseteq \PP_k^{g-1}$ is a rational normal curve of degree $g-1$, cf.~\cite{MR0463157}*{IV~\S5}.  Since $C$ is a rational normal curve, it is of genus $0$ and $I(C)$ is generated by $I_2(C)$.

Let $H$ be a hyperplane section of $C\subseteq \PP^{g-1}_k$.  Fix an integer $d\geq 2$.  We have $\deg(dH)=d(g-1) >0$ and hence $l(dH)=d(g-1)-0+1$ by Riemann--Roch.  The $d$-th component of the graded ring $k[x_1,\ldots, x_g]/I(C)$ has dimension $l(dH)$ and hence $\dim_k I_d(C) = \binom{g-1+d}{d}-l(dH) = \binom{g-1+d}{d}- d(g-1)-1$.  The claimed dimension for $I_2(C)$ is now immediate.
\end{proof}

\subsection{Computing $I_d(C)$}  \label{S:computing Id}

Fix notation and assumptions as in \S\ref{SS:setup}.  In this section, we describe how to compute $I_d(C)$ for a fixed integer $d\geq 0$.   We may assume that $d\geq 2$ since $I_0(C)=I_1(C)=0$.  

Let $\Gamma_G$ be the congruence subgroup of $\SL_2(\ZZ)$ from \S\ref{S:modular curves} and let $w$ be the width of $\Gamma_G$ at $\infty$.  Note that $w$ is the smallest positive integer for which $\left(\begin{smallmatrix}1 & w \\0 & 1\end{smallmatrix}\right)$ modulo $N$ lies in $G$.  The $q$-expansion of any cusp form $f \in S_k(\Gamma(N),\QQ(\zeta_N))^G \subseteq S_k(\Gamma_G,\CC)$ is a power series in $q_w:=e^{2\pi i \tau/w}$ since $f| \left(\begin{smallmatrix}1 & w \\0 & 1\end{smallmatrix}\right) =f$.   

Let $M_d$ be the set of monomials in $\QQ[x_1,\ldots, x_g]$ of degree $d$.   For each $m\in M_d$, we have $m(f_1,\ldots,f_g) \in S_{2d}(\Gamma(N),\QQ(\zeta_N))^G$ and hence
\[
m(f_1,\ldots, f_g) = \sum_{n=0}^\infty a_{m,n} q_w^{n}
\]
for unique $a_{m,n} \in \QQ(\zeta_N)$.

\begin{lemma}
Consider a homogeneous polynomial $F\in \QQ[x_1,\ldots, x_g]$ of degree $d$; we have $F= \sum_{m\in M_d} c_m m$ for unique $c_m\in \QQ$.  Then $F$ is an element of $I_d(C)$ if and only if 
\begin{align} \label{E:cm linear}
\sum_{m\in M_d} a_{m,n} \,c_m = 0
\end{align}
holds for all $0\leq n \leq d(2g-1)$.
\end{lemma}
\begin{proof}
From the injective homomorphism (\ref{E:graded isomorphism alpha}) of graded rings, we find that $F$ lies in $I_d(C)$ if and only if $F(f_1,\ldots, f_g)=0$.  We have 
\[
F(f_1,\ldots, f_g) = \sum_{m\in M_d} c_m m(f_1,\ldots, f_g)= \sum_{n=0}^\infty \big(\sum_{m\in M_d} a_{m,n} c_m\big) q_w^n.
\]
So $F$ lies in $I_d(C)$ if and only if (\ref{E:cm linear}) holds for all $n\geq 0$.  One implication of the lemma is now immediate.

Now suppose that $F \notin I_d(C)$; equivalently, $F(f_1,\ldots, f_g)\neq 0$.   Let $\nu$ be the smallest integer for which the coefficient of $q_w^\nu$ in the $q$-expansion of $F(f_1,\ldots, f_g)$ is non-zero.  It thus suffices to prove that $\nu\leq d(2g-1) $.

 The differential form $f_j(\tau) d\tau$ on $\Gamma_G\backslash \mathfrak{H}$ extends to a holomorphic differential $1$-form on $X_G(\CC)$ for each $1\leq j \leq g$.   Define $\omega:= F(f_1(\tau),\ldots, f_g(\tau))\, (d\tau)^{d}$.  Since $F$ is homogeneous of degree $d$, $F(f_1(\tau),\ldots, f_g(\tau))\, (d\tau)^{d}$ gives rise to a holomorphic differential $d$-form $\omega$ on $X_G(\CC)$.    We have $\omega\neq 0$ since $F(f_1,\ldots, f_g)\neq 0$.   

The divisor of $\omega$ is effective and has degree $d(2g-2)$.  Therefore, $v_P(\omega)\leq  d(2g-2)$, where $P\in X_G(\CC)$ is the cusp at infinity and $v_P(\omega)$ is the order of vanishing of $\omega$ at $P$.   One can verify that $v_P(\omega) = \nu -d$.   Therefore, $\nu \leq  d(2g-2) +d=d(2g-1)$. 
\end{proof}

By the above lemma, $I_d(C)$ consists of the polynomials $\sum_{m\in M_d} c_m m$ with $c_m\in \QQ$ such that (\ref{E:cm linear}) holds for all $0\leq n \leq d(2g-1)$.    So given the values $a_{m,n} \in \QQ(\zeta_N)$ with $m\in M_d$ and $0\leq n \leq d(2g-1)$, computing a basis of $I_d(C)$ is basic linear algebra (note that since $\QQ(\zeta_N)$ over $\QQ$ has basis $1,\zeta_N,\ldots, \zeta_N^{\varphi(N)-1}$, each equation (\ref{E:cm linear}) can be replaced by $\varphi(N)$ linear equations with rational coefficients).    \\

It remains to explain how $a_{m,n}$ can be computed for fixed $m\in M_d$ and $0\leq n \leq d(2g-1)$; recall that the $f_j$ are given by their $q$-expansions and we can compute an arbitrary number of terms.  The $q$-expansion of each cusp form $f_j$ lies in $q_w\cdot \QQ(\zeta_N)\brak{q_w}$.    Using this and that $m$ is homogeneous of degree $d$, one can check that $a_{m,n}$ is determined by the coefficients of $q_w^i$ in the $q$-expansion of $f_j$ for all  $0\leq i \leq n-d$ and $1\leq j \leq g$.    

In particular,  we deduce that $I_d(C)$ can be computed from the coefficients of $q_w^i$ in the $q$-expansion of $f_j$ for all  $0\leq i \leq d(2g-1)-d=d(2g-2)$ and $1\leq j \leq g$.    

\subsection{Computing the curve $C$}

Fix notation and assumptions as in \S\ref{SS:setup}.  To compute the curve $C\subseteq \PP^{g-1}_\QQ$, it suffices to find a set of generators of the ideal $I(C) \subseteq \QQ[x_1,\ldots, x_g]$.  We have assumed that $g\geq 2$.  We may further assume that $g\geq 3$ since $C=\PP^1_\QQ$ and $I(C)=0$ if $g=2$.

From \S\ref{S:computing Id}, we can compute the $\QQ$-vector space $I_2(C)$ from the cusps forms $f_1,\ldots, f_g$; let $F_1,\ldots, F_r$ be a basis.   By Lemmas~\ref{L:canonical map, not hyperelliptic} and \ref{L:canonical map, hyperelliptic}, and using $g\geq 3$, we find that $X_G$ is hyperelliptic if and only if $r=(g-1)(g-2)/2$.     If $X_G$ is hyperelliptic, Lemma~\ref{L:canonical map, hyperelliptic} implies that the curve $C$ has genus $0$ and the ideal $I(C)$ is generated by $F_1,\ldots, F_r$.  We may now assume that $X_G$ is not hyperelliptic.

Suppose $g=3$.  By Lemma~\ref{L:canonical map, not hyperelliptic}, $I_4(C)$ has dimension $1$ and generates the ideal $I(C)$.    From \S\ref{S:computing Id}, we can compute the $\QQ$-vector space $I_4(C)$; let $F$ be a basis.      The ideal $I(C)$ is thus generated by $F$ and hence $C \subseteq \PP^2$ is the smooth plane quartic defined by $F=0$.  

Now assume that $g\geq 4$.   By Lemma~\ref{L:canonical map, not hyperelliptic}, the ideal $I(C)$ is generated by $I_2(C)$ and $I_3(C)$, and $I_3(C)$ has dimension $(g - 3)(g^2 + 6g - 10)/6$.   Let $W$ be the subspace of $I_3(C)$ generated by $x_i F_j$ with $1\leq i \leq g$ and $1\leq j \leq r$.    If $W$ has dimension $(g - 3)(g^2 + 6g - 10)/6$, then $W=I_3(C)$ and hence $I(C)$ is generated by $F_1,\ldots, F_r$.

Finally suppose that the dimension of $W$ is not $(g - 3)(g^2 + 6g - 10)/6$.    From \S\ref{S:computing Id}, we can compute the $\QQ$-vector space $I_3(C)$.   Let $G_1,\ldots, G_s$ be polynomials in $I_3(C)$ that give rise to a basis in $I_3(C)/W$.   The ideal $I(C)$ is thus generated by $F_1,\ldots, F_r, G_1,\ldots, G_s$.   It is not needed for our purposes, but one can further show that $s=g-3$.

\begin{remark}
One can choose $f_1,\ldots, f_g$ to be a basis of the $\ZZ$-module $S_2(\Gamma(N),\QQ(\zeta_N))^G \cap S_2(\Gamma(N),\ZZ[\zeta_N])$.   We can choose $F_1,\ldots, F_r$ to be basis of the $\ZZ$-module $I_2(C)\cap \ZZ[x_1,\ldots, x_g]$.    The LLL-algorithm can be used to makes such choices with relatively small coefficients.
\end{remark} 
 
 \bibliographystyle{plain}

\end{document}